\documentclass[a4paper,11pt,draft]{amsart}
\usepackage[margin=30mm]{geometry}
\usepackage{amssymb, amsmath, amsthm, amscd}

\usepackage[T1]{fontenc}
\usepackage{eucal,mathrsfs,dsfont}%gives nicer set-names. Use \mathds instead of \mathds
\usepackage{color}%cyan,red;magenta,green;yellow,blue
\renewcommand{\le}{\leqslant}
\renewcommand{\ge}{\geqslant}

\definecolor{mno}{rgb}{0.5,0.1,0.5}

\newcommand{\R}{\mathds R}

\newcommand{\w}{\omega}

\newcommand{\Pp}{\mathds P}
\newcommand{\Ee}{\mathds E}

\newcommand{\I}{\mathds 1}

\newcommand{\Z}{\mathds Z}

\newcommand{\F}{\mathscr{F}}

\newcommand{\WEHI}{\textup{WEHI}}
\newcommand{\EHI}{\textup{EHI}}

\newcommand{\eps}{\varepsilon}

\newtheorem{theorem}{Theorem}[section]
\newtheorem{lemma}[theorem]{Lemma}
\newtheorem{proposition}[theorem]{Proposition}

\theoremstyle{definition}

\newtheorem{example}[theorem]{Example}
\newtheorem{remark}[theorem]{Remark}

\numberwithin{equation}{section}

\begin{document}
\allowdisplaybreaks
\title[Random conductance models with stable-like jumps]
{\bfseries
Random conductance models with stable-like jumps: Heat kernel estimates and Harnack inequalities}
\author{Xin Chen\qquad Takashi Kumagai\qquad Jian Wang}
\thanks{\emph{X.\ Chen:}
   Department of Mathematics, Shanghai Jiao Tong University, 200240 Shanghai, P.R. China. \texttt{chenxin217@sjtu.edu.cn}}
\thanks{\emph{T.\ Kumagai:}
 Research Institute for Mathematical Sciences,
Kyoto University, Kyoto 606-8502, Japan.
\texttt{kumagai@kurims.kyoto-u.ac.jp}}
  \thanks{\emph{J.\ Wang:}
    College of Mathematics and Informatics \& Fujian Key Laboratory of Mathematical Analysis and Applications, Fujian Normal University, 350007 Fuzhou, P.R. China. \texttt{jianwang@fjnu.edu.cn}}

\date{}

\maketitle

\begin{abstract}We establish two-sided heat kernel estimates
for random conductance models
with non-uniformly elliptic (possibly degenerate) stable-like jumps on graphs.
These are long range counterparts of
well known two-sided Gaussian heat kernel estimates
by M.T. Barlow for nearest neighbor (short range) random walks on the supercritical percolation cluster.
Unlike the cases for nearest neighbor conductance models, the idea
through parabolic Harnack inequalities
does not work, since even
elliptic Harnack inequalities do not hold in the present setting.
As an application, we establish
the local limit theorem for the models.

\medskip

\medskip

\end{abstract}
\allowdisplaybreaks

\section{Introduction}\label{section1}
Consider a bond percolation on $\mathbb Z^d$, $d\ge 2$; namely, on each nearest neighbor bond $x,y\in \mathbb Z^d$ with $|x-y|=1$, we put a random conductance $w_{x,y}$ in
such a way that $\{w_{x,y}(\omega): x,y\in \mathbb Z^d, |x-y|=1\}$ are i.i.d.\ Bernoulli so that
${\mathbb P}(w_{x,y}(\omega)=1)=p$ and ${\mathbb P}(w_{x,y}(\omega)=0)=1-p$ for some $p\in [0,1]$.
It is known that there exists a constant $p_c(\mathbb Z^d)\in (0,1)$ such that
almost surely there exists a unique infinite cluster ${\mathcal C}_\infty(\omega)$ (i.e.\ a connected component of bonds with conductance $1$) when $p>p_c(\mathbb Z^d)$ and no infinite cluster
when $p<p_c(\mathbb Z^d)$. Suppose $p>p_c(\mathbb Z^d)$ and
consider a continuous time simple random walk $(X^\omega_t)_{t\ge 0}$ on the infinite cluster.
Let
$p^\omega(t,x,y)$ be the heat kernel (or the transition density function) of $X^\omega$, i.e.,
\[
p^\omega(t,x,y):=\frac{\Pp^x\big(X_t^\omega=y\big)}{\mu_y},
\]
where $\mu_y$ is a number of bonds whose one end is $y$.
In the cerebrated paper \cite{Bar}, Barlow proved the following detailed heat kernel estimates
those are
almost sure w.r.t. the randomness of the environment; namely the
following quenched estimates:

There exist random variables $\{R_x(\omega)\}_{x\in {\mathbb Z^d}}$ with
$R_x(\omega)\in [1,\infty)$ for all $x\in {\mathcal C}_\infty(\omega)$ ${\mathbb P}$-a.s.\ $\omega$ and
constants $c_i=c_i(d,p)$, $i=1,\cdots, 4$ such that for all $x,y\in {\mathcal C}_\infty(\omega)$ with
$t\ge |x-y|\vee R_x(\omega)$, $p(t,x,y)$ satisfies the following
\begin{equation}\label{eq:GHKEMB}
c_1t^{-d/2}\exp (-c_2|x-y|^2/t)\le p^\omega(t,x,y) \le c_3t^{-d/2}\exp (-c_4|x-y|^2/t).
\end{equation}
Note that because of the degenerate structure of the (random) environment, we cannot expect
\eqref{eq:GHKEMB} to hold for all $t\ge 1$. Barlow's results assert that such a Gaussian estimate
holds as a long time estimate despite of the degenerate structure.

When the conductances are bounded from above and below (the uniformly
elliptic case) and the global volume doubling condition holds, it is well known that
the associated heat kernel for the nearest
neighbor conductance models
obeys two-sided Gaussian estimates (see e.g. \cite{CG,De}).
When the conductances are
non-uniformly elliptic, the situation becomes complex and delicate. The
supercritical bond percolation discussed above is a typical example.
(We note that Mathieu and Remy (\cite{MR}) also obtained a
large time on-diagonal heat kernel upper bound for this model.)

Barlow's results have many applications. For example, they were applied
crucially in the proofs of quenched local central limit theorem and quenched invariant
principle for the model.
The results have been extended to nearest neighbor
conductance models with ergodic media in \cite{ADS3,Sap}, and these large time heat
kernel estimates (and parabolic Harnack inequalities) have been key estimates
in the field of random conductance models, see
\cite{ABDH, ADS1,ADS2,ADS3,ADS4,ACDS,BChen, BD,BeB,BKM,BP,BsRod,M,MatP,SS,PRS,Sap} or
\cite{Bs, Kum} for the survey on these topics.

\medskip

In this paper, we consider quenched heat kernel estimates for random conductance models
that allow big jumps. In particular, we establish two-sided heat kernel estimates
for random conductance models with non-uniformly elliptic and possibly degenerate
stable-like jumps on graphs. Despite of the fundamental importance of the problem,
so far there are only a few results for conductance models with long range jumps. As far as
we are aware, this is the first work on detailed heat kernel estimates for possibly degenerate
random walks with long range jumps. We now explain our framework and a result.

\bigskip

Suppose that $G=(V,E_V)$ is a locally finite
connected infinite graph,
where $V$ and $E_V$ denote the collection of vertices and edges respectively.
For $x\neq y\in V$, we write $\rho(x,y)$ for the graph distance, i.e., $\rho(x,y)$ is the smallest positive length of a path (that is, a sequence $x_0=x,$ $x_1$, $\cdots$, $x_l=y$ such that $(x_i,x_{i+1})\in E_V$ for all $0\le i\le l-1$) joining $x$ and $y$. We set $\rho(x,x)=0$ for all $x\in V$.
Let
$B(x,r):=\{y\in V: \rho(y,x)\le r\}$ denote the ball with center $x\in V$ and radius $r\ge1$.
Let $\mu$ be a measure on $V$ such that the following assumption holds.

\smallskip

\paragraph{Assumption {\bf ($d$-Vol)}} \label{dvol}  {\it There are constants $R_0\ge1$, $\kappa>0$,
$c_\mu\ge 1$, $\theta\in (0,1)$ and $d>0$
$($all four are independent of $R_0$$)$
such that the following hold:
\begin{eqnarray}
\label{e:a1-1}0<\mu_x\le c_{\mu},&\quad \forall x\in V,\\
\label{e:a1-2}\inf_{x\in B(0,R)}\mu_x\ge R^{-\kappa},&\quad \forall R\ge R_0,\\
\label{a2-2}
~c_\mu^{-1}r^d\le \mu(B(x,r))\le c_\mu r^d,&\quad \forall
R\ge R_0, \forall x\in B(0,6R) \mbox{ and }
R^{\theta}/2\le \forall r \le 2R.
\end{eqnarray}}

\vskip -0.5cm

\noindent In particular,
$(G,\mu)$ only satisfies
the $d$-set condition uniformly for large scale
under \eqref{a2-2}.
For $p\ge1$, let $L^p(V;\mu)=\{f\in
\R^V:\sum_{x\in V}|f(x)|^p\mu_x<\infty\}$, and $\|f\|_p$ be the
$L^p(V;\mu)$ norm of $f$ with respect to $\mu$. Let
$L^\infty(V;\mu)$ be the space of bounded measurable functions on
$V$, and $\|f\|_\infty$ be the $L^\infty(V;\mu)$ norm of $f$.

Suppose that $\{w_{x,y}:x,y\in V\}$ is a sequence such that
$w_{x,y}\ge0$ and $w_{x,y}=w_{y,x}$ for all $x\neq y$, and
\begin{equation}\label{eq:condwxy}
\sum_{y \in V: y\neq x}\frac{w_{x,y}\mu_y}
{\rho(x,y)^{d+\alpha}}
<\infty,\quad x\in V,
\end{equation} where $\alpha\in (0,2)$. For simplicity, we set $w_{x,x}=0$ for all $x\in V$.
We
can define a regular Dirichlet form $(D,\mathscr{F})$ as follows
(see the first statement in \cite[Theorem 3.2]{CKK})
\begin{equation}\label{e1-1}
\begin{split}
 D(f,f)&=\frac{1}{2}\sum_{x,y\in V}(f(x)-f(y))^2\frac{w_{x,y}}{\rho(x,y)^{d+\alpha}}\mu_x\mu_y,\\
\F&=\{f\in L^2(V;\mu): D(f,f)<\infty\}.
\end{split}
\end{equation}
It is easy to verify that the infinitesimal generator ${\mathcal L}$ associated with $(D,\F)$ is given by
$$ {\mathcal L} f(x)=
\sum_{z\in V}(f(z)-f(x))\frac{w_{x,z}\mu_z}{\rho(x,z)^{d+\alpha}}.
$$
Let $X:=(X_t)_{t\ge 0}$ be the symmetric Hunt process associated with $(D,\F)$.
When
$\mu$ is a counting measure on $G$ (resp.\ $\mu_x$ is chosen to be satisfied that $1=\sum_{z\in V} \frac{w_{x,z}\mu_z}{\rho(x,z)^{d+\alpha}}$ for all $x\in V$),
 the associated process $X$ is
 called the variable speed random walk (resp. the constant speed random walk) in the literature.

 For any subset $D\subset V$, let
$\tau_D:=\inf\{t>0: X_t \notin D\}$ be the first exit time from $D$ for the process $X$. Denote by $X^D:=(X_t^D)_{t\ge0}$ the Dirichlet process, i.e.,
$$
X^D_t:=\begin{cases}
X_t,\quad\text{if}\ t<\tau_D,\\
\,\partial,\,\,\quad \text{if}\ t\ge \tau_D,
\end{cases}
$$
where $\partial$ denotes the cemetery point.
Let
$p^D(t,x,y)$ be the Dirichlet heat kernel associated with the process $X^D$.

In the accompanied paper \cite{CKW}, we discussed the quenched invariance principle
for conductance models with stable-like jumps.
As a continuation of \cite{CKW}, we consider heat kernel estimates for the conductance models.
The main difficulty here is due to that neither that conductances
are uniformly elliptic (possibly degenerate)
nor the global
$d$-set condition is supposed
to be satisfied. To illustrate our contribution, we state the following
result for random conductance models on $\mathbb{L}:=\mathbb{Z}^{d_1}_+\times\mathbb{Z}^{d_2}$ with
$d_1, d_2\in \mathbb{Z}_+:=\{0,1,2,\cdots\}$ such that $d_1+d_2\ge1$ (i.e., $V=\mathbb{L}$ and
the coefficients $w_{x,y}$ given in \eqref{eq:condwxy} are random
variables).
\begin{theorem}\label{t1}{\bf(Heat kernel estimates for Variable speed random walks)}
Let $V=\mathbb{L}$
with $d>4-2\alpha$, and $\{w_{x,y}(\w): x,y\in \mathbb{L}\}$
be a sequence of independent random variables on some probability space
$(\Omega, \F_\Omega, \Pp)$ such that for any $x\neq y$, $w_{x,y}=w_{y,x}\ge 0$
and
\begin{align*}
& \sup_{x,y \in \mathbb{L}:x\neq y}\Pp\left(w_{x,y}=0\right)<2^{-4},\quad \sup_{x,y\in \mathbb{L}:x \neq y}\left(\Ee \big[w_{x,y}^{p}\big]+\Ee
\big[w_{x,y}^{-q}\I_{\{w_{x,y}>0\}}\big]\right)<\infty
\end{align*}
for some $p,q\in \Z_+$ with
$$p>\max\big\{{{(d+1+\theta_0)}/({d\theta_0}), (d+1)}/(2\theta_0(2-\alpha))\big\},\quad
q>{(d+1+\theta_0)}/({d\theta_0}),$$ where
$\theta_0:=\alpha/(2d+\alpha)$. Let $(X_t^{\w})_{t\ge 0}$ be the
symmetric Hunt process corresponding to the Dirichlet form
$(D,\mathscr{F})$ above with random variables $\{w_{x,y}(\w): x \neq
y\in \mathbb{L}\}$ and $\mu$ being the counting measure on $\mathbb{L}$.
Denote by $p^\w(t,x,y)$ the heat kernel of the process
$(X_t^{\w})_{t\ge0}$. Then, $\Pp$-a.s.\ $\w\in \Omega$, for any
$x\in \mathbb{L}$,
there is a constant $R_x(\w)\ge1$ such that for all
$R>R_x(\w)$ and for all $t>0$ and $y\in \mathbb{L}$ with $t\ge (|x-y|\vee R_x(\w))^{\theta\alpha}$,
\begin{equation}\label{eq:HKjumpstable}
C_1\left(t^{-d/\alpha}\wedge \frac{t}{|x-y|^{d+\alpha}}\right)\le p^{\w}(t,x,y)
\le C_2\left(t^{-d/\alpha}\wedge \frac{t}{|x-y|^{d+\alpha}}\right),
\end{equation}
where $\theta\in (0,1)$ and $C_1,C_2>0$ are constants
independent of $R_x(\w)$, $t$, $x$ and $y$.
\end{theorem}

Note that \eqref{eq:HKjumpstable} is the typical heat kernel estimates for stable-like jumps,
and it corresponds to the Gaussian estimates \eqref{eq:GHKEMB} for the nearest neighbor cases.

\medskip

Let us explain some related work. As we mentioned above,
there are only a few
results for conductance models
with stable-like (long range) jumps.
When the conductances are
uniformly elliptic and the global volume doubling condition holds, heat kernel estimates
like \eqref{eq:HKjumpstable} have been discussed, for instance in \cite{BBK, BL, MS0,MS}.
In these aforementioned papers, a lot of arguments are heavily based on
uniformly elliptic conductances, and two-sided pointwise bounds of
conductances are also necessary and frequently used. The
corresponding results for non-local Dirichlet forms on general
metric measure spaces now have been obtained, and, in particular, the De Giorgi-Nash-Moser
theory are developed, see \cite{CK1,CK2,CKW1, CKW2, CKW3, GHH1,GHH2}
and the references therein. Crawford and Sly \cite{CS0} proved
on-diagonal heat kernel upper bounds for random walks on
the infinite cluster of supercritical long range percolation, see \cite{CrS} for the scaling limit of random walks on long range percolation clusters.
Due to
the singularity of long range percolation cluster, off-diagonal heat kernel
estimates are still unknown and seem to be quite different from those for
conductance models with stable-like jumps.
As mentioned before, it does not seem that
heat kernel estimates for conductance models with non-uniformly
elliptic stable-like jumps and under non-uniformly volume doubling
condition are available till now. In this paper we will address
this problem completely.

\medskip

We summarize
some difficulties of our problem as follows.
\begin{itemize} \!\!
\item[(i)] As for nearest neighbor  non-uniformly elliptic conductance models, a usual (and powerful) idea is to establish first elliptic and parabolic Harnack inequalities, and then deduce heat kernel bounds. For example, see \cite{ADS1,ADS2,ADS3} for the recent study on ergodic environments of nearest neighbor random conductance models under some integrability conditions. However,
    we will prove that in the present setting
elliptic Harnack inequalities do not hold even for large balls
and so parabolic Harnack inequalities
do not hold in general either,
    when conductances are not uniformly elliptic. This is totally different from uniformly elliptic stable-like jumps, see e.g.\ \cite{BBK, BL, MS0,MS}, or uniformly
    elliptic stable-like jumps with variable orders on the Euclidean space $\R^d$, see e.g. \cite{BK}. We refer readers to Proposition \ref{t5-1} and
   Example \ref{exmephi} below
    for details.

\item[(ii)] In case of nearest neighbor  non-uniformly elliptic conductance models,
off-diagonal upper bounds of
the heat kernel
can be deduced from on diagonal upper bounds using the maximum principle initiated by
Grigor'yan on manifolds \cite{Gr} and developed in \cite{Fo} on
graphs, see e.g.\ the proofs of \cite[Proposition 1.2]{BKM} or
\cite[Proposition 3.3]{BChen}. Because of the effect of long range
jumps, such approach
does not seem to be
applicable in our model.

\item[(iii)] As mentioned before,
in order to establish
heat kernel estimates for uniformly elliptic stable-like jumps,
pointwise upper and lower bounds of conductances are crucially used
in \cite{BBK, BL, MS0,MS}. In particular, uniform
    lower bounds of conductances yield
    Nash/Sobolev
       inequalities for the associated Dirichlet form, which in turn imply on-diagonal
    heat kernel upper bounds
        immediately. Furthermore, based on
    Nash/Sobolev inequalities, the Davies method
        was adopted in \cite{BBK, BL, MS0,MS}
    to derive off-diagonal upper bound estimates for heat kernel. However, in the setting of our paper
    Nash/Sobolev
      inequalities do not hold, and so
    the  approaches above are not applicable.
    \end{itemize}

 To establish two-sided heat kernel estimates for long range and non-uniformly elliptic
 conductance models with stable-like jumps, we will apply the localization argument for Dirichlet heat
 kernel estimates, and then pass
 through
 these to global heat kernel estimates via the Dynkin-Hunt formula.
 For this, we make full use of estimates for the exit time of the process obtained in \cite{CKW}.
 Though part of
 ideas in the proofs
 are motivated by the study of global heat kernel
 estimates for uniformly elliptic conductance models with stable-like jumps (for instance, see \cite{BBK,BGK}),
 it seems that this is
 the first time to adopt them to investigate the corresponding Dirichlet heat kernel estimates for large time scale.
 In the proof,
 a lot of non-trivial modifications and new ideas are
 required. Actually,
  in this paper we will establish heat kernel estimates under
 a quite general framework beyond Theorem \ref{t1},
see Theorem \ref{t2-2} and Theorem \ref{t2-3} below. In particular, only the $d$-set condition with large
scale {\bf ($d$-Vol)} and locally summable conditions on conductances (see Assumptions {\bf
(HK1)}--{\bf (HK3)} below) are assumed. These conditions can be regarded as a
generalization of \lq\lq good ball\rq\rq \, conditions for nearest
neighbor conductance models in \cite{BChen} into long range
conductance models. As an application of a series of (large scale)
probability estimates for exit times and regularity of parabolic
harmonic functions as well as heat kernel estimates for large time, we can also justify the local limit theorem for
our model, see Theorem \ref{p3-2} below.

 \ \

The organization of this paper is
as follows. The next section is devoted to heat kernel estimates for large time. This part is split into three subsections. We first consider on-diagonal upper bounds, later study off-diagonal upper bounds, and then lower bound estimates. In Section \ref{sec4}, we present some estimates for Green functions, and also give a consequence of elliptic Harnack inequalities. In the last section, we apply our previous results to random conductances with stable-like jumps.

\section{Heat Kernel Estimates: Large Time}
\label{section2}
To obtain heat kernel estimates for large time, we need the following three assumptions on $\{w_{x,y}: x,y \in V\}$. We fix $0\in V$,
and define $B_z^w(x,r):=\{y\in B(x,r): w_{y,z}>0\}$ for all $x,z\in V$ and $r>0$. Set $B^w(x,r):=B_x^w(x,r)$  for simplicity.

\smallskip

 \paragraph{Assumption {\bf (HK1)}} \label{as2-1}
{\it Suppose that there exist $R_0\ge 1$, $\theta\in (0,1)$, $c_0>1/2$ and
$C_1>0$ $($all three
are independent of $R_0)$ such that
\begin{itemize}
 \item[(i)] For
every $R>R_0$ and $R^{\theta}/2\le r \le 2R$,
\begin{equation*}\label{a2-2-1}
\sup_{x\in B(0,6R)}\sum_{y\in V:\rho(x,y)\le  r}
\frac{w_{x,y}\mu_y}
{\rho(x,y)^{d+\alpha-2}}\le C_1 r^{2-\alpha},
\end{equation*}
\begin{equation*}\label{a2-2-1a}
\mu(B_z^w(x,r))\ge c_0\mu(B(x,r)),\quad x,z\in B(0,6R),
\end{equation*} and
\begin{equation*}\label{a2-2-2}
\sup_{x \in B(0,6R)}
\sum_{y \in B^w(x,c_* r)}
w_{x,y}^{-1}\mu_y\le C_1r^{d},
\end{equation*}
where  $c_*:= 8c_\mu^{2/d}$.
\item[(ii)] For every $R>R_0$ and $r\ge  R^{\theta}/2$,
\begin{equation*}\label{a2-2-3}
\begin{split}
\sup_{x \in B(0,6R)}\sum_{y \in V:
\rho(x,y)>r}\frac{w_{x,y}\mu_y}{\rho(x,y)^{d+\alpha}}\le C_1r^{-\alpha}.
\end{split}
\end{equation*}
\end{itemize}
}

\ \

 \paragraph{Assumption {\bf (HK2)}}\label{a3-1a}
{\it Suppose that for some fixed $\theta\in (0,1)$, there exist
$R_0\ge1$ and $C_2>0$ $($independent of $R_0)$ such that for every
$R>R_0$ and $R^{\theta}/2\le r \le 2R$,
\begin{equation*}\label{a3-1-1}
\sup_{x,y\in B(0,6R)}\sum_{z\in V:\rho(z,y)\le  r}
w_{x,z}\mu_z\le C_2 r^{d}.
\end{equation*}}

 \paragraph{Assumption {\bf (HK3)}}\label{a3-1b}
{\it Suppose that for some fixed $\theta\in (0,1)$, there exist
$R_0\ge1$ and $C_3>0$ $($independent of $R_0)$ such that for every
$R>R_0$ and $R^{\theta}/2\le r \le 2R$,
\begin{equation*}\label{a3-1-1-1}
\inf_{x,y\in B(0,6R)}
\sum_{z\in V:\rho(z,y)\le  r}
w_{x,z}\mu_z\ge C_3 r^{d}.
\end{equation*}}

{Assumption \bf{(HK1)}} is a slight modification of \cite[Assumption {\bf(Exi.)}]{CKW}, which was used to derive the distribution and the expectation of exit time, see \cite[Theorem 3.4]{CKW} or Theorem \ref{p2-1} below. Actually, in \cite{CKW} {Assumption \bf{(HK1)}} is also adopted to yield the (large scale) H\"{o}lder regularity of associated parabolic functions, see \cite[Theorem 3.8]{CKW} or Theorem \ref{p2-2} in Subsection \ref{section2.3}. We further note that when $\alpha\in(0,1)$, one can replace Assumption {\bf{(HK1)}} by \cite[Assumption {\bf(Exi'.)}]{CKW} to deduce the distribution and the expectation of exit time, which may refine some conditions in Assumption {\bf{(HK1)}}. The details are left to interested readers.

\subsection{Upper bounds of the heat kernel}

\subsubsection{{\bf On diagonal upper bounds of the heat kernel}}
 \begin{proposition}{\bf (On diagonal upper bound for the Dirichlet heat kernel)}\label{p-on}
 Assume that Assumption {\bf($d$-Vol)} holds with some $\theta \in (0,1)$, $\kappa>0$ and $R_0\ge 1$,
  and that there exists $R'_0\ge1$ such that for all
 $R\ge R'_0$ and $R^{\theta}\le r \le R$,
 \begin{equation}\label{t2-0}
\mu(B_z^w(x,r))\ge c_0\mu(B(x,r)),\quad x,z\in B(0,2R)
\end{equation}
and
 \begin{equation}\label{t2-1}
 \sup_{x\in B(0,2R)}
 \sum_{y\in B^w(x,2r)}
 w_{x,y}^{-1}\mu_y\le C_0r^d,
 \end{equation}
where $c_0>1/2$ and $C_0>0$ are independent of $R_0$, $R'_0$, $R$ and $r$.
 Then, for every $\theta'\in (\theta,1)$, there exists a constant $R_1\ge 1$ such that for all $R>R_1$, $x_1,x_2 \in B(0,R)$ and
 $t\ge R^{\theta' \alpha}$,
  \begin{equation}\label{t4-2-1a}
 p^{B(0,R)}(t,x_1,x_2)\le C_1t^{-d/\alpha},
 \end{equation}
 where $C_1$ is a positive constant independent of
$R_0$, $R'_0$, $R_1$, $R$, $x_1$, $x_2$ and $t$.
 \end{proposition}
 \begin{proof}
 The proof is
 to some extent
 similar to that of \cite[Proposition 2.2]{CKW}, which is
concerned with on diagonal upper bounds for global heat kernel of truncated processes. We will provide the complete proof here
for convenience of readers. Noticing that, by the Cauchy-Schwarz inequality, $p(2t,x_1,x_2)\le p(t,x_1,x_1)^{1/2}p(t,x_2,x_2)^{1/2}$ for any $x_1,x_2\in V$ and $t>0$, it suffices to show \eqref{t4-2-1a} for the
case $x_1=x_2$. We will split the proof into three steps.

\noindent {\bf Step (1)} We first prove that there are constants $R_2\ge 1$ and $C_2>0$ (independent of $R_2$) such that for any $R>R_2$, $x\in
B(0,R)$, $R^{\theta}\le r< R$ and any measurable function $f$ on $V$,
\begin{equation}\label{t4-2-2}
\sum_{z\in B(x,r)}(f(z)-(f)_{B^w(z,r)})^2\mu_z\le C_2 r^{\alpha}
\sum_{z\in B(x,r),y\in B(x,2r)}(f(z)-f(y))^2 \frac{w_{z,y}}{\rho(z,y)^{d+\alpha}}\mu_y\mu_z,
\end{equation} where for $A\subset V$,
$$(f)_A:=\frac{1}{\mu(A)}\sum_{z\in A}f(z)\mu_z.$$

Indeed, for every $R>R_2:=R_0^{1/\theta}\vee R'_0$ with $R_0$ and $R'_0$ being the constants
in Assumption {\bf($d$-Vol)}  and Proposition  \ref{p-on}
respectively,  $x \in B(0,R)$ and $R^{\theta}\le r \le R$, we have
\begin{align*}
&\sum_{z\in B(x,r)}(f(z)-(f)_{B^w(z,r)})^2\mu_z\\
&=\sum_{z\in B(x,r)}\left(\frac{1}{\mu(B^w(z,r))}\sum_{y\in B^w(z,r)}(f(z)-f(y))\mu_y\right)^2\mu_z\\
&\le \frac{c_1}{r^{2d}}\sum_{z\in B(x,r)}
\Bigg[\bigg(\sum_{y\in B^w(z,r)}(f(z)-f(y))^2 \frac{w_{z,y}\mu_y}{\rho(z,y)^{d+\alpha}}\bigg)
\bigg(\sum_{y\in B^w(z,r)}w_{z,y}^{-1}\rho(z,y)^{d+\alpha}\mu_y\bigg)\Bigg]\mu_z \\
&\le c_2r^{-d+\alpha}\left(\sup_{z \in B(0,2R)}
\sum_{y \in B^w(z,2r)}
w_{z,y}^{-1}\mu_y\right)\left(\sum_{z\in B(x,r),y\in B(x,2r)}
\big(f(z)-f(y)\big)^2\frac{w_{z,y}\mu_y\mu_z}{\rho(z,y)^{d+\alpha}}\right)\\
&\le c_3r^{\alpha}\sum_{z\in B(x,r),y\in B(x,2r)}
\big(f(z)-f(y)\big)^2\frac{w_{z,y}}{
\rho(z,y)^{d+\alpha}}\mu_y\mu_z,
 \end{align*}
 where in the first inequality we used \eqref{a2-2}, \eqref{t2-0} and the Cauchy-Schwarz inequality, and the third
 inequality is due to
 \eqref{t2-1}.

 {\bf Step (2)} For any $x\in B(0,R)$ and $R>R_2$, let $f_t(z)=p^{B(0,R)}(t,x,z)$ and $\psi(t)=p^{B(0,R)}(2t,x,x)$
 for all $z\in V$ and $t\ge 0$. Then,
 $\psi(t)=\sum_{z\in B(0,R)} f_t(z)^2\mu_z$ and
 $$\psi'(t)= 2\sum_{z\in B(0,R)} \frac{d\, f_t(z)}{dt} f_t(z)\mu_z=-\sum_{z,y\in V}(f_t(z)-f_t(y))^2
 \frac{w_{z,y}}{\rho(z,y)^{d+\alpha}}\mu_y\mu_z.$$
In particular, the second equality above, yielded by the integration by parts formula, holds since the finite summation
over variable $z$ together with \eqref{eq:condwxy} ensures the integrability of the associated terms.

Let $2R^{\theta}\le r(t)\le R$ be a constant to be determined later.
Let $B(x_i, r(t)/2)$ ($i=1,\cdots, m$) be the maximal collection of disjoint balls with centers in $B(0,R)$.
Set $B_i=B(x_i,r(t))$ and $B_i^*=B(x_i, 2r(t))$ for $1\le i \le m$.
Note that
$B(0,R)\subset \bigcup_{i=1}^mB_i\subset B(0,R+r(t));$
moreover, if $x\in B(0,R+r(t))\cap B_i^*$ for some $1\le i\le m$, then $B(x_i,r(t)/2)\subset B(x,3r(t))$. So
$$c_4r(t)^d\ge \mu(B(0,3r(t)))\ge \sum_{i=1}^m\I_{\{x\in B_i^*\}} \mu(B(x_i,r(t)/2))\ge c_5r(t)^d\sharp\{i:x\in B_i^*\},$$
where $\sharp A$ is
a number of elements
in the set $A$
for any $A\subset \Z$, and  we used \eqref{a2-2} and the fact that $r(t)\ge 2r_0$.
Thus, any
$x\in B(0,R+r(t))$ is in at most $c_6:=c_4/c_5$ of the ball $B_i^*$
(hence at most $c_6$ of the ball $B_i$), and
\begin{equation}\label{t4-2-2a}
\begin{split}
\sum_{i=1}^m\sum_{z\in B_i}=&\sum_{i=1}^m\sum_{z\in
B(0,R+r(t))}\I_{B_i}(z)
=\sum_{z\in
B(0,R+r(t))}\sum_{i=1}^m\I_{B_i}(z) \le  c_6\sum_{z \in
B(0,R+r(t))}.\end{split}
\end{equation}
Noting that $R>R_2$ and $2R^\theta\le r(t)\le R$, we obtain
\begin{align*}
&\sum_{z,y\in V} (f_t(z)-f_t(y))^2\frac{w_{z,y}}{\rho(z,y)^{d+\alpha}}\mu_y\mu_z\\
&\ge \frac{1}{c_6}\sum_{i=1}^m\sum_{z\in B_i}\sum_{y\in B_i^*}(f_t(z)-f_t(y))^2\frac{w_{z,y}}{\rho(z,y)^{d+\alpha}}
\mu_y\mu_z\\
&\ge \frac{c_7}{r(t)^{\alpha}}\left[\sum_{i=1}^m \sum_{z\in B_i}
f_t^2(z)\mu_z-
2\sum_{i=1}^m \sum_{z\in B_i}f_t(z)(f_t)_{B^w(z,r(t))}\mu_z\right]=: \frac{c_7}{r(t)^{\alpha}}(I_1-I_2), \end{align*}
where the first inequality is due to \eqref{t4-2-2a} and
in the second inequality we used \eqref{t4-2-2}.

Furthermore, we have
\begin{align*}I_1\ge &\sum_{z\in \cup_{i=1}^mB_i} f_t^2(z)\mu_z\ge\sum_{z\in B(0,R)} f_t^2(z)\mu_z
=\psi(t).
\end{align*}
  Note that $2R^{\theta}\le r(t)\le R$.  According to \eqref{a2-2}, \eqref{t2-0} and the
 fact that $\sum_{z\in V}f_t(z)$ $\mu_z\le 1$, we have
 \begin{align*}
 \sup_{z \in B(0,2R)}(f_t)_{B^w(z,r(t))}\le
\sup_{z\in B(0,2R)}\mu\big(B^w(z,r(t))\big)^{-1}\cdot \sum_{z\in V}f_t(z)\mu_z\le c_7^*r(t)^{-d}.
 \end{align*}
 Hence, by \eqref{t4-2-2a},
 \begin{align*}
 I_2&\le
 c_7^* r(t)^{-d}\sum_{i=1}^m\sum_{z \in B_i} f_t(z)\mu_z
 \le c_5c_7^*r(t)^{-d}\sum_{z \in B(0,R+r(t))}f_t(z)\mu_z\le c_5c_7^*r(t)^{-d}.
 \end{align*}
 Therefore, combining with all the estimates above, we obtain that for every $2R^{\theta}\le r(t)\le R$ with $R>R_2$,
\begin{equation}\label{t4-2-3}
\psi'(t)\le -c_8r(t)^{-\alpha}\left(\psi(t)-c_9r(t)^{-d}\right).
\end{equation}

{\bf Step (3)}
For any $\theta'\in (\theta,1)$ and any $R>R_2$ large enough, we claim that there exists
$t_0\in [R^{\theta\alpha}, R^{\theta'\alpha}]$ such that
\begin{equation}\label{t4-2-4}
\left( \frac{1}{2c_{9}}\psi(t_0)\right)^{-1/d}\ge 2R^{\theta}.
\end{equation}
Indeed, assume that \eqref{t4-2-4} does not hold. Then, for all $R^{\theta\alpha}\le t\le R^{\theta'\alpha}$,
\begin{equation}\label{t4-2-5}
\left( \frac{1}{2c_{9}}\psi(t)\right)^{-1/d}< 2R^{\theta},
\end{equation}
which implies that $\psi(t)\ge 2c_{9}(2R^{\theta})^{-d}$ for all $R^{\theta\alpha}\le t\le R^{\theta'\alpha}$.
Therefore, taking $r(t)=2R^{\theta}$ in \eqref{t4-2-3}, we find that for all $R^{\theta\alpha}\le t\le R^{\theta'\alpha}$,
$$\psi'(t)\le -\frac{c_{8}}{2}(2R^\theta)^{-\alpha} \psi(t),$$
which along with the fact
 $\psi(t)\le \mu_x^{-1}\le R^
\kappa$ for all $t>0$ and $x\in B(0,R)$ (due to \eqref{e:a1-2})
yields that for all $R^{\theta\alpha}\le t\le R^{\theta'\alpha}$,
$$\psi(t)\le  R^
\kappa e^{-\frac{c_{8}}{2}(2R^{\theta})^{-\alpha}(t-R^{\theta\alpha})}.$$
In particular,
$$\psi(R^{\theta'\alpha})\le R^
\kappa e^{-\frac{c_{8}}{2}(2R^{\theta})^{-\alpha}(R^{\theta'\alpha}-R^{\theta\alpha})}.$$

On the other hand, by \eqref{t4-2-5}, we have
$\psi(R^{\theta'\alpha})\ge
2c_{9} (2R^{\theta})^{-d}.$
Thus, there is a contradiction between these two inequalities above for $R$ large enough. In particular,
there exists $R_1>R_2$ such that \eqref{t4-2-4} holds for all $R>R_1$.

From now on, we may and do assume
that \eqref{t4-2-4} holds for all $R>R_1$. Since $t\mapsto\psi(t)$ is non-increasing on $(0,\infty)$
and $t_0\le R^{\theta'\alpha}$,  we have that for all $R^{\theta'\alpha}\le t\le R^{\alpha},$
$$
\left( \frac{1}{2c_{9}}\psi(t)\right)^{-1/d}\ge 2R^{\theta}.
$$
Let $$\tilde{t}_0:=\sup\bigg\{t>0: \left(
\frac{1}{2c_{9}}\psi(t)\right)^{-1/d}<R/2\bigg\}.$$ By the
non-increasing property of $\psi$ on $(0,\infty)$ again,
if $\tilde t_0\le 2R^{\theta'\alpha}$, then for $2R^{\theta'\alpha}\le t\le R^{\alpha}$
\begin{align*}
\psi(t)&\le \psi(\tilde t_0)=2c_{9}(R/2)^{-d}\le
c_{10}t^{-d/\alpha}.
\end{align*}

If $\tilde t_0>2R^{\theta'\alpha}$, then
$$
2R^{\theta}\le \left( \frac{1}{2c_{9}}\psi(t)\right)^{-1/d}\le  R/2
$$ for all  $R^{\theta'\alpha}\le t\le \tilde t_0$.
Taking
$r(t)=\big( \frac{1}{2c_{9}}\psi(t)\big)^{-1/d}$ in
\eqref{t4-2-3}, we know that for all $R^{\theta'\alpha}\le t\le \tilde t_0$,
$
\psi'(t)\le
-c_{11}\psi(t)^{1+\alpha/d}.
$
Hence, for all $2R^{\theta'\alpha}\le s\le
\tilde t_0$,
\begin{align*}
\psi(s)& \le
c_{12}\Big(s-R^{\theta'\alpha}+\psi(R^{\theta'\alpha})^{-\alpha/d}\Big)^{-d/\alpha}\le
c_{13}s^{-d/\alpha}.
\end{align*}
If $2R^{\theta'\alpha}\le\tilde t_0\le R^{\alpha}$,
then for all $\tilde t_0 \le s\le R^{\alpha}$, we have
$$
\psi(s)\le \psi(\tilde t_0)=2c_{9}(R/2)^{-d}\le c_{15}s^{-d/\alpha}.
$$
If $\tilde t_0>R^{\alpha}$, then \eqref{t4-2-1a} holds
similarly
for every $2R^{\theta'\alpha}\le t \le R^{\alpha}$.
Combining
all the estimates above and choosing $\theta'$ larger if necessary, we can obtain \eqref{t4-2-1a} for all $R>R_1$ and
$2R^{\theta'\alpha}\le t \le R^{\alpha}$.

Finally, for every $x \in B(0,R)$ and $t>R^{\alpha}$,
taking $N\ge R$ such that $x \in B(0,N)$ and
$2N^{\theta'\alpha}\le t \le N^{\alpha}$, we can get
\begin{equation*}
p^{B(0,R)}(t,x,x)\le p^{B(0,N)}(t,x,x)\le C_1 t^{-d/\alpha},
\end{equation*}
which implies \eqref{t4-2-1a} also holds for all $t>R^{\alpha}$.
Altogether, we obtain \eqref{t4-2-1a} for $t\ge 2R^{\theta' \alpha}$
for all $R>R_1$.
By changing the choice of $R_1$ (namely taking $2^{1/(\theta' \alpha)}R_1$
as a new $R_1$), we have \eqref{t4-2-1a} for $t\ge R^{\theta' \alpha}$
for all $R>R_1$, and the proof is complete.
\end{proof}

The following statement is an improvement of \cite[Theorem 3.4]{CKW}, which was proven under the
global $d$-set condition.
\begin{theorem}\label{p2-1}
Suppose that
Assumptions {\bf($d$-Vol)}  and {\bf(HK1)} hold with some constants
$\theta \in (0,1)$ and $R_0\ge1$. Then, for every $\theta'\in
(\theta,1)$, there exist constants $\delta \in (\theta,1)$,
$R_1\ge1$ such that
the following hold for all $R>R_1$ and $R^{\delta}\le r \le R$,
\begin{eqnarray}\label{p2-1-1}
&\sup_{x \in B(0,2R)}\Pp_x\big(\tau_{B(x,r)}\le C_0r^{\alpha}\big)\le
{1}/{4},\\
\label{p2-1-2}
&\begin{split}
\sup_{x \in B(0,2R)}\Pp_x\big(\tau_{B(x,r)}\le t \big)\le
C_1\Big(\frac{ t}{r^{\alpha}}\Big)^{1/2}\Big[1\vee \log\Big(\frac{r^{\alpha}}{t}\Big)\Big],\quad t\ge r^{\theta'\alpha},
\end{split}\\
\label{l2-2-1}
&\begin{split}
C_2r^{\alpha}\le \inf_{x \in B(0,2R)}\Ee_x\big[\tau_{B(x,r)}\big]
\le \sup_{x \in B(0,2R)}\Ee_x\big[\tau_{B(x,r)}\big]\le
C_1r^{\alpha},
\end{split}
\end{eqnarray} where
$C_0$, $C_1$ and $C_2>0$ are
independent of $R_0$, $R_1$, $R$, $r$ and $t$.
\end{theorem}
\begin{proof} The proof is heavily motivated by that of \cite[Theorem 3.4]{CKW}, and we only present main different points here.

First, it is seen from the proof of \cite[Theorem 3.4]{CKW} that the crucial point for the required assertions is to verify moment estimates (see \cite[Proposition 2.3]{CKW}) for the truncation of localized processes under Assumptions {\bf($d$-Vol)} and {\bf(HK1)}.
A difference from \cite[Section 2.2]{CKW} is,
here we will adopt the localization approach by using reflected Dirichlet forms on bounded sets.
In details, we consider the following localization of truncated reflected
Dirichlet form on the ball $B(0,6R)$:
\begin{align*}
\hat D^{R,R}(f,f)=&\sum_{x,y\in
B(0,6R): \rho(x,y)\le R}\big(f(x)-f(y)\big)^2\frac{ w_{x,y}}{\rho(x,y)^{d+\alpha}}\mu_x\mu_y,\quad
f\in
\hat \F^{R,R},\\
\hat \F^{R,R}=&\{f \in L^2(B(0,6R);\mu):
\hat D^{R,R}(f,f)<\infty\}.
\end{align*}
Let $(\hat X_t^{R, R})_{t\ge0}$ be the Hunt process associated with $(\hat D^{R,R},\hat \F^{R, R})$.
Regard $B(0,6R)$ as the whole space $V$ in \cite[Section 2.2]{CKW}.
By carefully tracking the proofs of \cite[Proposition 2.2 and Proposition 2.3]{CKW} and noticing that the lower bound
of $\mu_x$ was not used in the proofs, we can prove that, under Assumptions {\bf($d$-Vol)} and {\bf (HK1)}, for every $\theta'\in (\theta,1)$, there exist $R_1\ge 1$ and $C_3>0$ (independent of $R_1$)
such that for all
$x\in B(0,6R)$,
$$
\Ee_x\big[\rho\big(\hat X_t^{R, R},x\big)\big]\le C_3R\left(\frac{t}{R^\alpha}\right)^{1/2}
\left[1+\log\left(\frac{R^{\alpha}}{t}\right)
\right],\quad R^{\theta' \alpha}\le t \le R^{\alpha}.
$$
This along with the proof of \cite[Proposition 3.2]{CKW} further yields that for all $x_0\in B(0,2R)$ and $t\ge R^{\theta'\alpha}$
\begin{equation}\label{p2-1-4}
\Pp_x\big(\hat \tau_{B(x_0,R)}^{R,R}\le t\big)\le C_4\left(\frac{t}{R^\alpha}\right)^{1/2}
\left[1+\log\left(\frac{R^{\alpha}}{t}\right)
\right],
\end{equation}
where $\hat \tau_D^{R,R}$ denotes the first exit time from $D\subset B(0,6R)$ for the process
$(\hat X_t^{R,R})_{t\ge 0}$.

Next, we define the truncated Dirichlet form $(D^R, \F^R)$ as follows
\begin{align*}
D^{R}(f,f)=&\sum_{x,y\in
V: \rho(x,y)\le R}\big(f(x)-f(y)\big)^2\frac{ w_{x,y}}{\rho(x,y)^{d+\alpha}}\mu_x\mu_y,\quad
f\in
\F^{R},\\
\F^{R}=&\{f \in L^2(V;\mu):
D^{R}(f,f)<\infty\}.
\end{align*}
Let $(X^R)_{t\ge 0}$ be the Hunt process associated with $(D^R, \F^R)$.
Then, it is not difficult to verify that
for any $x_0\in B(0,2R)$ and $t>0$,
\begin{equation}\label{p2-1-3a}
\Pp_{x_0}\big(\tau^R_{B(x_0,R)}\le
t\big)=\Pp_{x_0}\big(\hat\tau^{R,R}_{B(x_0,R)}\le
t\big),
\end{equation}
where $\tau_D^{R}$ denotes the first exit time from $D\subset V$ for the process
$(X_t^{R})_{t\ge 0}$.

Therefore, putting \eqref{p2-1-4}, \eqref{p2-1-3a}, \cite[Lemma 3.1]{CKW} and Assumption ({\bf HK1}) (ii) together, we find that
for all $x_0\in B(0,2R)$ and $t>0$,
$$
\Pp_{x_0}\big(\tau_{B(x_0,R)}\le
t\big)\le C_5\left(\frac{t}{R^\alpha}\right)^{1/2}
\left[1+\log\left(\frac{R^{\alpha}}{t}\right)\right].
$$
Hence, the desired assertion follows from this estimate and the argument of
\cite[Theorem 3.4]{CKW}.
\end{proof}

We now prove the global on-diagonal
upper bound.
\begin{proposition}\label{t4-2}{\bf (On diagonal upper bound for the heat kernel)}
Suppose that Assumptions {\bf($d$-Vol)} and {\bf (HK1)} hold with some constants $\theta \in (0, {\alpha}/({2d+\alpha}))$
and $R_0\ge1$.
Then, for any $\theta'\in (\theta, {\alpha}/({2d+\alpha}))$, there exists a constant $R_1\ge 1$ such that for all $x,y \in V$ and
$t\ge \big(R_1\vee \rho(0,x)\vee \rho(0,y)\big)^{\theta'\alpha}$,
 \begin{equation}\label{t4-2-1}
p(t,x,y)\le C_1(1\vee \mu_y^{-1})t^{-d/\alpha},
 \end{equation}
where $C_1>0$ is a constant independent of $R_0$, $R_1$, $x$, $y$
and $t$.
In particular,
for any $\theta'\in (\theta, {\alpha}/({2d+\alpha}))$, there exists a constant $T_0>0$ such that for all
$t>T_0$ and $x,y \in B(0,t^{1/(\theta'\alpha)})$, \eqref{t4-2-1} holds with constant $C_1>0$ independent of $R_0$, $T_0$, $x$, $y$ and $t$.
\end{proposition}
\begin{proof}
According to the Dynkin-Hunt formula,
for every $N\ge 2R\ge1$, $x,y\in B(0,R)$ and $t>0$,
\begin{align*}
p(t,x,y)&=p^{B(0,N)}(t,x,y)+
\Ee_x\big[p\big(t-\tau_{B(0,N)},X_{\tau_{B(0,N)}},y\big)\I_{\{t>\tau_{B(0,N)}\}}\big]\\
&=:J_{1,N}(t)+J_{2,N}(t).
\end{align*}
According to \eqref{t4-2-1a} and \eqref{p2-1-2}, for any $\varepsilon\in (0,1/2)$ and $\theta'\in (\theta,1)$, we can find a constant $R_1\ge1$ large enough such that for every
$N>2R_1$, $x \in B(0,R)$  and $t\ge N^{\theta'\alpha}$,
\begin{equation}\label{t4-2-3a}
J_{1,N}(t)\le c_1t^{-d/\alpha}\end{equation}
and \begin{equation}\label{t4-2-3-b}
J_{2,N}(t)\le \mu_y^{-1}\Pp_x (\tau_{B(0,N)}<t)\le \mu_y^{-1}\Pp_x(\tau_{B(x,N/2)}<t)
\le c_1(tN^{-\alpha})^{(1/2)-\varepsilon}\mu_y^{-1},
\end{equation}
where in the inequality above we used the facts that $p(t,x,y)\le \mu_y^{-1}$
for all $x,y\in V$,
and  $B(x,N/2)\subset B(0,N)$ for any $x \in B(0,R)$ with $2R\le N$.

Let $N_0(t):=[t^{({2d}/({\alpha^2(1-2\varepsilon))})+({1}/{\alpha})}]$. For any $\theta\in
(0,{\alpha}/({2d+\alpha}))$ and $\theta'\in (\theta, $ ${\alpha}/({2d+\alpha}))$, we take $\varepsilon \in (0,1/2)$
such that
$$ \left(\frac{2d}{\alpha^2(1-2\varepsilon)}+\frac{1}{\alpha}\right)\theta'\alpha
=1.$$
Then,  for all $R>R_1$ and $t\ge(3R)^{\theta'\alpha}$ with $R>R_1$,
they hold that $t\ge N_0(t)^{\theta'\alpha}$ and $N_0(t)\ge 2R>2R_1$.
So, taking $N=N_0(t)$ in \eqref{t4-2-3a} and \eqref{t4-2-3-b}, we obtain that  for all
$R>R_1$, $t\ge(3R)^{\theta'\alpha}$ and $x,y \in B(0,R)$,
\begin{align*}
p(t,x,y)\le & J_{1,N_0(t)}(t)+J_{2,N_0(t)}(t)\le c_1(1\vee \mu_y^{-1})\Big(t^{-d/\alpha}+\big(tN_0(t)^{-\alpha}\big)^{(1/2)-\varepsilon}\Big)\\
\le & c_2(1\vee \mu_y^{-1})t^{-d/\alpha}.
\end{align*}
Changing $\theta'$ a little large if necessary, without loss of generality we may and can assume that the estimate above holds for all $t\ge R^{\theta'\alpha}$.
Therefore, \eqref{t4-2-1} follows immediately by choosing $R=R_1\vee \rho(0,x)\vee \rho(0,y)$.
Furthermore,
choosing $T_0=R_1^{\theta'\alpha}$ and $R=t^{1/(\theta'\alpha)}$ respectively in the conclusion above, we can get the second assertion.
\end{proof}

\subsection{Off diagonal upper bound estimates}
We first recall the following L\'evy system formula, see e.g.  \cite[Lemma 4.7]{CK1} or \cite[Appendix A]{CK2}.
\begin{lemma}\label{levy}
For any $x\in V$, stopping time $\tau$ $($with respect to the natural filtration of the precess $X$$)$, and non-negative measurable function $f$ on $[0,\infty)\times V\times V$ with $f(s, z, z)=0$ for all $z\in V$ and $s\ge 0$, we have
\begin{equation}\label{e:levy}
\Ee_x \left[\sum_{s\le \tau} f(s,X_{s-}, X_s) \right] = \Ee_x
\left[ \int_0^\tau \left(\sum_{z\in V} f(s,X_s, z) \frac{w_{X_s,z}\mu_z}{\rho(X_s,z)^{d+\alpha}} \right)\, ds \right].
\end{equation}
\end{lemma}

\begin{lemma}\label{L:up2}
Suppose that Assumption {\bf (HK2)} holds with constants $\theta\in (0,1)$ and $R_0\ge1$.
Then there exists a constant $R_1\ge1$ such that for all
$R>R_1$, $R^{\theta}\le s \le r/2\le R$, $x_0\in B(0,R)$, $x\in B(x_0, r/2)$, $y \in B(x_0,2r)^c \cap
B(0,R)$ and $t>0$,
\begin{equation}\label{l4-2-1}
\Pp_x(\tau_{B(x_0,r)}\le t, X_{\tau_{B(x_0,r)}}\in B(y,s))\le \frac{C_1 t s^d}{r^{d+\alpha}},
\end{equation}
where $C_1>0$ is independent
of $x_0$, $x$, $R_0$, $R_1$, $R$, $s$ and $t$.
\end{lemma}
\begin{proof}
By  Assumption {\bf (HK2)}, we know that there exists $R_2\ge1$ such that for all
$R>R_2$, $x,y\in B(0,2R)$ and $R^{\theta}\le s \le R$,
\begin{equation}\label{l4-2-2}
\sum_{z\in V: \rho(y,z)\le s}w_{x,z}\mu_z\le c_1s^d.
\end{equation}
Then, according to \eqref{e:levy},
for any $x_0\in B(0,R)$,
$x\in B(x_0, r/2)$ and $y\in B(x_0,2r)^c \cap B(0,R)$, $R^{\theta}\le s \le r/2 \le R$ and $t>0$,
\begin{align*}
\Pp_x(\tau_{B(x_0,r)}\le t, X_{\tau_{B(x_0,r)}}\in B(y,s))
&\le \Ee_x\left[\int_0^{t\wedge \tau_{B(x_0,r)}}\sum_{u\in V:\rho(u,y)\le s}\frac{w_{X_v,u}\mu_u}{\rho(X_v,u)^{d+\alpha}}\,dv\right]\\
&\le c_2t r^{-d-\alpha}\sup_{z,y\in B(0,2R)}\sum_{u\in V:\rho(u,y)\le s}w_{z,u}\mu_u\\
&\le c_3ts^d r^{-d-\alpha},
\end{align*}
where in the second inequality we used the fact that
$$\rho(u,v)\ge \rho(y,v)-\rho(y,u)\ge \rho(y,x_0)-\rho(x_0,v)-\rho(y,u)\ge
r/2$$ for all $v \in B(x_0,r)$ and $u \in B(y,s)$,
and the last inequality is due to \eqref{l4-2-2}.
\end{proof}

\begin{proposition}\label{T:ofde}{\bf(Off diagonal upper bound for the Dirichlet heat kernel)}
Suppose that Assumptions {\bf($d$-Vol)}, {\bf (HK1)} and {\bf (HK2)} hold with $\theta\in (0,1)$ and $R_0\ge1$.
Then, for every $\theta'\in (\theta,1)$, there exists a constant $R_1\ge 1$ such that for all
$R>R_1$, $x,y \in B(0,R)$ and $t\ge R^{\theta'\alpha}$,
\begin{equation}\label{t4-3-1}
p^{B(0,R)}(t,x,y)\le C_1
\left(t^{-d/\alpha}\wedge\frac{t}{\rho(x,y)^{d+\alpha}}\right),
\end{equation}
where $C_1>0$ is independent of $x$, $y$, $R_0$, $R_1$, $R$ and $t$.
\end{proposition}

\begin{proof} We follow the proof of \cite[Theorem 1.2 (b)$\Rightarrow$ (a)]{BGK} with some required modifications
due to the large scale setting. In the proof below
the constant $c$ will be changed from line to line, and will be independent of
$x$, $y$, $R_0$, $R_1$, $R$ and $t$.

For any $q\in[0,\infty)$ and $N \ge 1$, we call $(H_{q,N})$ as follows:
\begin{itemize} \it
\item[$(H_{q,N})$] there is a constant $c>0$ such that for all $R\ge R_1$,
and $x,y\in B(0,R)$,
$$p^{B(0,R)}(t,x,y)\le ct^{-d/\alpha}\left(\frac{t}{\rho(x,y)^\alpha}\right)^{q},\quad
t\ge  NR^{\theta'\alpha}.$$
\end{itemize}
In particular,  by \eqref{t4-2-1a}, $(H_{0,N})$ holds for all $N\ge 1$. Now, we will prove that
\begin{itemize}
\item [(i)] If $(H_{q,N})$ holds with some $0\le q<d/\alpha$, then $(H_{q+\delta,2N})$ holds for every $\delta\in (0,1/2)$.
\item [(ii)] If $(H_{q,N})$ holds with some $q>d/\alpha$, then $(H_{(1+(d/\alpha))\wedge ( q+\delta),2N})$ holds
for every $\delta\in (0,1/2)$.
\end{itemize}  Suppose that (i) and (ii) hold true. Since the iteration from $q=0$ and $q=1+(d/\alpha)$ only takes finite times, we can get  \eqref{t4-3-1} by taking $\theta'$ a little bit larger.
In the following, we will prove (i) and (ii) respectively.

{\bf Step (1)} We assume that $(H_{q,N})$ holds with $0\le q<d/\alpha$. Let $t\ge 2NR^{\theta'\alpha}$,
and $x,y \in B(0,R)$. If $\rho(x,y)\le 8t^{1/\alpha}$, then, by $(H_{0,N})$, $(H_{q+\delta,2N})$ holds for every $\delta\in (0,1/2)$.
Next, we suppose that $\rho(x,y)> 8t^{1/\alpha}$. Set $\rho_0=t^{1/\alpha}$ and $r=\rho(x,y)/2$, so that $r>4\rho_0$.
Applying \cite[Lemma 2.1]{BGK} to the Dirichlet semigroup $(P^{B(0,R)}_t)_{t\ge0}$ with $U=B(x,r)\cap B(0,R)$ and $V=B(y,r)\cap B(0,R)$, we obtain that for all non-negative measurable functions $f$ and $g$ on $V$ with supports contained in $B(0,R)$,
\begin{equation}\label{e:pup1}\begin{split}
\langle P_t^{B(0,R)}f, g\rangle\le &\langle
\Ee_{\cdot} \big[\I_{\{\tau_{B(x,r)}\le t/2\}} P_{t-\tau_{B(x,r)}}^{B(0,R)}f(X_{\tau_{B(x,r)}})\big],g\rangle\\
&+\langle \Ee_{\cdot} \big[\I_{\{\tau_{B(y,r)}\le t/2\}} P_{t-\tau_{B(y,r)}}^{B(0,R)}g(X_{\tau_{B(y,r)}})\big],f\rangle,
\end{split}
\end{equation}
where $\langle \cdot, \cdot \rangle$ denotes the inner product on $L^2(V;\mu)$.
Let $f$ be supported in $B(y,\rho_0)\cap B(0,R)$ and $g$ be supported in $B(x,\rho_0)\cap B(0,R)$. Then, it holds that
\begin{align*}
&\langle \Ee_{\cdot}\big[ \I_{\{\tau_{B(x,r)}\le t/2\}} P_{t-\tau_{B(x,r)}}^{B(0,R)}f(X_{\tau_{B(x,r)}})\big],g\rangle\\
&=\sum_{z\in B(x,\rho_0)\cap B(0,R)} \Ee_z\big[ \I_{\{\tau_{B(x,r)}\le t/2\}} P_{t-\tau_{B(x,r)}}^{B(0,R)}f(X_{\tau_{B(x,r)}})\big] g(z)\mu_z.
\end{align*}
A similar equality holds for the second term in  the right hand side of \eqref{e:pup1}.

Below, we write $\tau=\tau_{B(x,r)}$ and $B={B(0,R)}$ for simplicity.
Set $\rho_k=2^k\rho_0$ for $k\ge1$, and consider the annuli
$$
A_1:=B(y,\rho_1), \quad A_k:=B(y,\rho_k)\backslash B(y,\rho_{k-1}),\quad k\ge 2.
$$
Then, for every $z \in B(x,\rho_0)$,
\begin{align*}  &\Ee_z \left[\I_{\{\tau\le t/2\}} P^{B}_{t-\tau}f(X_{\tau})\right]=\sum_{k=1}^\infty
\Ee_z \left[\I_{\{\tau\le t/2\}} P^{B}_{t-\tau}f(X_{\tau})\I_{\{X_{\tau}\in A_k\}}\right]=:\sum_{k=1}^\infty E_k. \end{align*}

For $k\ge2$, note that if $X_\tau\in A_k$, then $\rho(X_{\tau},y)\ge \rho_{k-1}.$ So, for all $v\in B(y,\rho_0)$,
$$\rho(X_\tau,v)\ge \rho_{k-1}-\rho_0\ge \frac{1}{2}\rho_{k-1}= \frac{1}{4} \rho_{k}.$$
Recall that, for $\tau\le t/2$ and $t\ge 2NR^{\theta'\alpha}$, it holds that $NR^{\theta'\alpha}\le t/2\le t-\tau$.
Hence, by $(H_{q,N})$, if $X_{\tau}\in A_k$ for all
$k\ge 2$, then
$$
P_{t-\tau}^{B} f(X_\tau)=\sum_{v\in B(y,\rho_0)}p^{B}(t-\tau,X_\tau,v)f(v)\mu_v\le
ct^{-d/\alpha}\left( \frac{t}{\rho_k^\alpha}\right)^q\|f\|_1.$$
According to $(H_{0,N})$ and the fact that $\rho_1=2t^{1/\alpha}$, it is easy to see the inequality above also holds true
for $k=1$.

Next, we separately estimate the terms with $\rho_k>r/2$ and with $\rho_k\le r/2$.
Using the facts that $\rho_0<r/4$, $t/2\ge NR^{\theta'\alpha}$ and $r=\rho(x,y)/2> 4t^{1/\alpha}\ge 4 R^{\theta'\alpha}$,
we obtain from \eqref{p2-1-2} that for all $z\in B(x,\rho_0)\cap B(0,R)\subseteq B(0,R)$ and $\delta\in(0,1/2)$,
$$\Pp_z(\tau\le t/2)\le \Pp_z(\tau_{B(z,r/2)}\le t/2)\le c\Big(\frac{t}{r^\alpha}\Big)^{\delta}.$$
Hence, according to all these estimates above, for any $q>0$,
\begin{align*}
\sum_{k: 2\rho_k>r}E_k&\le\!\! c\sum_{k: 2\rho_k>r} \Pp_z(\tau\le t/2) t^{-d/\alpha}\left( \frac{t}{\rho_k^\alpha}\right)^q\|f\|_1\le\!\!  c\sum_{k: 2\rho_k>r} \Big(\frac{t}{r^\alpha}\Big)^{\delta} t^{-d/\alpha}\left( \frac{t}{\rho_k^\alpha}\right)^q\|f\|_1\\
&\le c\Big(\frac{t}{r^\alpha}\Big)^{\delta} t^{-d/\alpha}\left( \frac{t}{r^\alpha}\right)^q\|f\|_1\le ct^{-d/\alpha}\left( \frac{t}{r^\alpha}\right)^{q+\delta}\|f\|_1.
\end{align*} When $q=0$,
$$
\sum_{k: 2\rho_k>r}E_k\le ct^{-d/\alpha}\|f\|_1\Pp_z(\tau\le t/2)\le c t^{-d/\alpha}\left( \frac{t}{r^\alpha}\right)^{\delta}\|f\|_1.
$$

On the other hand, for $2\rho_k\le r$,  it holds that $2N^{1/\alpha}R^{\theta'}\le \rho_1\le \rho_k\le r/2\le R$.
Then, by  \eqref{l4-2-1},
for all $k\ge 1$ and $z \in B(x,\rho_0)\subseteq B(x,r/2)$,
$$
\Pp_z (\tau\le t/2, X_{\tau}\in A_k)\le \frac{ct\rho_k^d}{r^{d+\alpha}}.
$$
Combining this with
$(H_{q,N})$ yields
\begin{align*}
\sum_{{k}: 2\rho_k\le r} E_k &\le c\sum_{k: 2\rho_k\le r}
\Pp_z(\tau\le t/2, X_\tau\in B(y,\rho_k))
t^{-d/\alpha}\left( \frac{t}{\rho_k^\alpha}\right)^q\|f\|_1\\
&\le  c\sum_{k: 2\rho_k\le r} \frac{t\rho_k^d}{r^{d+\alpha}}
\frac{1}{t^{d/\alpha}}\left( \frac{t}{\rho_k^\alpha}\right)^q\|f\|_1\le ct^{-d/\alpha} \frac{t^{1+q}}{r^{d+\alpha}}\|f\|_1\sum_{k: 2\rho_k\le r} \rho_k^{d-\alpha q} \\
&\le c t^{-d/\alpha}\left( \frac{t}{r^\alpha}\right)^{1+q}\|f\|_1,
\end{align*} where in the last inequality we used the fact that
$\sum_{k: 2\rho_k\le r} \rho_k^{d-\alpha q}\le cr^{d-\alpha q}$
due to $q<d/\alpha$.

Thus, according to all the estimates above, we obtain that for any $R>R_1$, $\delta\in (0,1/2)$,
$t\ge2NR^{\theta'\alpha}$, $r>4t^{1/\alpha}$ and $z\in B(x,\rho_0)\cap B(0,R)$,
$$\Ee_z\left[\I_{\{\tau\le t/2\}} P^{B}_{t-\tau}f(X_{\tau})\right] \le c t^{-d/\alpha}
\left( \frac{t}{r^\alpha}\right)^{q+\delta}\|f\|_1$$
and so
$$\langle \Ee_{\cdot} \I_{\{\tau\le t/2\}} P_{t-\tau}^{B}f(X_{\tau}),g\rangle\le
c t^{-d/\alpha}\left( \frac{t}{r^\alpha}\right)^{q+\delta}\|f\|_1\|g\|_1.$$
Estimating similarly the second term in the right hand side of \eqref{e:pup1}, we finally get that
for all $R>R_1$, $\delta\in (0,1/2)$, $r>4t^{1/\alpha}$ and $t\ge2NR^{\theta'\alpha}$,
$$\langle P_t^{B}f, g\rangle\le ct^{-d/\alpha}
\left( \frac{t}{r^\alpha}\right)^{q+\delta}\|f\|_1\|g\|_1,$$
which yields that $(H_{q+\delta,2N})$ holds. So (i) has been shown.

Now we turn to (ii). Similarly, it suffices to consider the case $r>4t^{1/\alpha}$. Suppose that $(H_{q,N})$ holds for some $q>d/\alpha$ and $N\ge 1$. Then, following the argument above and carefully tracking
the constants, we arrive at that for all $R>R_1$, $\delta\in (0,1/2)$, $r>4t^{1/\alpha}$ and $t\ge 2NR^{\theta'\alpha}$,
$$\sum_{k: 2\rho_k>r} E_k\le  c t^{-d/\alpha}
\left( \frac{t}{r^\alpha}\right)^{q+\delta}\|f\|_1$$ and
$$\sum_{k: 2\rho_k\le r} E_k\le
ct^{-d/\alpha}\frac{t^{1+q}}{r^{d+\alpha}}\|f\|_1\sum_{k: 2\rho_k\le r} \rho_k^{d-\alpha q}
\le ct^{-d/\alpha}
\left( \frac{t}{r^\alpha}\right)^{1+(d/\alpha)}\|f\|_1,
$$
where in the last inequality we used the fact
$\sum_{k: 2\rho_k\le r} \rho_k^{d-\alpha q}\le c\rho_0^{d-\alpha q}\le ct^{d/\alpha-q}$, thanks to $q>d/\alpha$.
These estimates together imply that when $q>d/\alpha$, for all $R>R_1$, $\delta\in (0,1/2)$, $r>4t^{1/\alpha}$ and $t\ge2NR^{\theta'\alpha}$,
$$\langle \Ee_{\cdot} \I_{\{\tau\le t/2\}} P_{t-\tau}^{B}f(X_{\tau}),g\rangle\le
ct^{-d/\alpha}
\left( \frac{t}{r^\alpha}\right)^{(1+(d/\alpha))\wedge ( q+\delta)}\|f\|_1\|g\|_1.$$
To estimate the second term in the right hand side of \eqref{e:pup1} similarly, we know that for all $R>R_1$, $\delta\in (0,1/2)$, $r>4t^{1/\alpha}$ and $t\ge 2NR^{\theta'\alpha}$,
$$\langle P_t^{B}f, g\rangle\le
ct^{-d/\alpha}
\left( \frac{t}{r^\alpha}\right)^{(1+(d/\alpha))\wedge ( q+\delta)}\|f\|_1\|g\|_1.$$
So $(H_{(1+(d/\alpha))\wedge ( q+\delta)),2N})$ holds. Thus, we prove (ii) and so the proof is complete.
\end{proof}

By using Proposition \ref{T:ofde}, we can establish the following off diagonal upper bounds for heat kernel $p(t,x,y)$.

 \begin{proposition}{\bf (Off diagonal upper bounds for the heat kernel)}\label{T:ofe}
Suppose that Assumptions {\bf($d$-Vol)}, {\bf (HK1)} and {\bf (HK2)} hold with $\theta\in (0,
{\alpha}/({2d+\alpha}))$ and $R_0\ge 1$.
Then, for every $\theta'\in (\theta,{\alpha}/({2d+\alpha}))$, there is a constant
$R_1\ge1$ such that for any
$x,y \in V$ and $t>0$ with $$\rho(x,y)\ge (R_1\vee \rho(0,x)\vee \rho(0,y))^{{\alpha(1+\theta')}/({2(d+\alpha)})}$$ and
$$\rho(x,y)^{2\theta'(d+\alpha)/({1+\theta'})}\le t \le \rho(x,y)^{2(d+\alpha)}(R_1\vee \rho(0,x)\vee \rho(0,y))^{-\alpha},$$ we have
\begin{equation}\label{t4-4-1}
p(t,x,y)\le C_1(1\vee \mu_y^{-1})\frac{t}{\rho(x,y)^{d+\alpha}},
\end{equation}
where $C_1>0$ is a positive constant
independent of  $R_0$, $R_1$, $R$,  $x$, $y$ and $t$.
\end{proposition}
\begin{proof}
Similar to the proof of Proposition \ref{t4-2}, we apply the Dynkin-Hunt formula and obtain that for every $N>R\ge1$ and $x,y \in B(0,R)$,
\begin{align*}
p(t,x,y)&=p^{B(0,N)}(t,x,y)+
\Ee_x\big[p\big(t-\tau_{B(0,N)},X_{\tau_{B(0,N)}},y\big)\I_{\{t>\tau_{B(0,N)}\}}\big]\\
&=:J_{1,N}(t)+J_{2,N}(t).
\end{align*}
According to \eqref{t4-3-1} and \eqref{p2-1-2}, for any $\varepsilon\in(0,1/2)$ and $\theta_1\in (\theta,\theta')$, there exists a constant $R_1\ge1$ such that for every
$N\ge 2R>2R_1$, $x, y\in B(0,R)$ and $t\ge N^{\theta_1\alpha}$,
\begin{equation}\label{t4-4-2}
J_{1,N}(t)\le \frac{c_1t}{\rho(x,y)^{d+\alpha}}\end{equation} and
\begin{equation}\label{t4-4-22}
J_{2,N}(t)\le \mu_y^{-1}\Pp_x (\tau_{B(0,N)}<t)\le
\mu_y^{-1}\Pp_x (\tau_{B(x,N/2)}<t)\le c_2\mu_y^{-1}(tN^{-\alpha})^{(1/2)-\varepsilon},
\end{equation} where in the inequality above we used again the facts that $p(t,x,y)\le \mu_y^{-1}$ for all $x,y\in V$, and  $B(x,N/2)\subset B(0,N)$ for any $x \in B(0,R)$ with $2R\le N$.

Set $N(t,x,y):=\left[t^{-1/({\alpha(1-2\varepsilon)})}\rho(x,y)^{{2(d+\alpha)}/(
{\alpha(1-2\varepsilon)})}\right].$  Let $\theta'\in (\theta,{\alpha}/({2d+\alpha}))$ and  $\theta_1\in (\theta,\theta')$. For any
$x,y\in B(0,R)$  and $t>0$ with $\rho(x,y)\ge (3R)^{{\alpha(1+\theta')}/({2(d+\alpha)})}$ and
$$\rho(x,y)^{2\theta'(d+\alpha){}/({1+\theta'})}\le t \le \rho(x,y)^{2(d+\alpha)}(3R)^{-\alpha},$$
we can choose $\varepsilon>0$ such that
$$\frac{\theta'}{1+\theta'}>\frac{\theta_1}{1-2\varepsilon +\theta_1},$$ and so $N(t,x,y)\ge 2R$ and $t\ge  N(t,x,y)^{\theta_1\alpha}$ for $R$ large enough.
Note that $\rho(x,y)\ge (3R)^{{\alpha(1+\theta')}/({2(d+\alpha)})}$ implies
$\rho(x,y)^{2(d+\alpha)}(3R)^{-\alpha}\ge \rho(x,y)^{2\theta'(d+\alpha){}/({1+\theta'})}$.
Then, applying $N=N(t,x,y)$ into \eqref{t4-4-2} and \eqref{t4-4-22}, we can obtain that for any
$x,y\in B(0,R)$ and $t>0$ with $\rho(x,y)\ge (3R)^{{\alpha(1+\theta')}/({2(d+\alpha)})}$ and
$\rho(x,y)^{2\theta'(d+\alpha){}/({1+\theta'})}\le t \le \rho(x,y)^{2(d+\alpha)}(3R)^{-\alpha},$ it holds
$$
p(t,x,y) \le J_{1,N(t,x,y)}(t)+J_{2,N(t,x,y)}(t)\le   \frac{c_3(1\vee \mu_y^{-1})t}{\rho(x,y)^{d+\alpha}},
$$ where in the last inequality we used the fact that
$$N\ge t^{-({1+2\varepsilon})/({\alpha(1-2\varepsilon)})}\rho(x,y)^{{2(d+\alpha)}/
{(\alpha(1-2\varepsilon))}}.$$ Therefore, we prove the second desired estimates in \eqref{t4-4-1} by taking
$R= R_1\vee \rho(0,x)\vee \rho(0,y)$ and $\theta'$ a little bit larger.
\end{proof}

Finally, according to Propositions \ref{t4-2} and \ref{T:ofe}, we can
summarize the
following upper bound for the heat kernel $p(t,x,y)$.
\begin{theorem}{\bf  (Upper bound for the heat kernel)}\label{t2-2}
Suppose Assumptions {\bf($d$-Vol)}, {\bf (HK1)} and {\bf (HK2)} hold with $\theta\in (0,
 {\alpha}/({2d+\alpha}))$ and $R_0\ge1$. Then, for every
$\theta'\in (\theta_0,1)$ with $\theta_0:=\max\{2\theta(d+\alpha)/(\alpha({1+\theta})),  {\alpha}/({2d+\alpha})\}$,
 there is a constant
$R_1\ge1$ such that for any $R>R_1$, $x,y\in V$ with
$t>(R_1\vee \rho(0,x)\vee \rho(0,y))^{\theta'\alpha}$,
\begin{equation}\label{t2-2-1}
p(t,x,y)\le C_1(1\vee \mu_y^{-1})\Big(t^{-d/\alpha}\wedge\frac{t}{\rho(x,y)^{d+\alpha}}\Big),
\end{equation}
where $C_1>0$ is
independent of $R_0$, $R_1$, $R$, $x$, $y$ and $t$.
\end{theorem}
\begin{proof}
For any $x,y\in V$, let $D(x,y)=R_1\vee \rho(0,x)\vee \rho(0,y)$.
We first consider the case that $\rho(x,y)\ge D(x,y)^{{\alpha}/({2d+\alpha})}$.
 Below, we set $\theta'_0= 2\theta(d+\alpha)/(\alpha({1+\theta}))$. Note that, for any $\theta'\in (\theta'_0,1)$, we can find $\theta_1\in (\theta, {\alpha}/({2d+\alpha}))$ such that
$\theta'=2\theta_1(d+\alpha)/(\alpha({1+\theta_1}))$.
It follows from the facts $\theta_1< {\alpha}/({2d+\alpha})$ and
$ D(x,y)^{{\alpha}/({2d+\alpha})}\le \rho(x,y)\le 2D(x,y)$ that
$
\rho(x,y)\ge D(x,y)^{{\alpha(1+\theta_1)}/({2(d+\alpha)})},
$
and $$
\rho(x,y)^{2\theta_1(d+\alpha)/({1+\theta_1})}=
\rho(x,y)^{\theta'\alpha}\le (2D(x,y))^{\theta'\alpha},
$$ as well as
$$\rho(x,y)^{2(d+\alpha)}D(x,y)^{-\alpha}\ge \rho(x,y)^{\alpha}.$$ These along with
\eqref{t4-4-1} yield that for any $D(x,y)^{\theta'\alpha}\le t \le \rho(x,y)^{\alpha}$ (increasing $\theta'$ a little larger if necessary),
$$
p(t,x,y)\le \frac{c_2(1\vee \mu_y^{-1})t}{\rho(x,y)^{d+\alpha}}.
$$
On the other hand, note that $\theta<\theta'_0<\theta'$.
According to
\eqref{t4-2-1},
\begin{equation}\label{t2-2-2}
p(t,x,y)\le c_3(1\vee \mu_y^{-1})t^{-d/\alpha},\quad t>D(x,y)^{\theta'\alpha}.
\end{equation}
Combing both estimates above yields \eqref{t2-2-1} for the case that $\rho(x,y)\ge D(x,y)^{{\alpha}/({2d+\alpha})}$.

Next, we consider the case that $\rho(x,y)\le D(x,y)^{{\alpha}/({2d+\alpha})}$.  Since $\theta'>\theta_0\ge \alpha/(2d+\alpha)$, it holds that
\begin{equation*}
\rho(x,y)\le D(x,y)^{{\alpha}/({2d+\alpha})} \le D(x,y)^{\theta'}.
\end{equation*}
Hence, \eqref{t2-2-1} follows from \eqref{t2-2-2} for the case that $\rho(x,y)\le D(x,y)^{{\alpha}/({2d+\alpha})}$.
Therefore, we prove the desired assertion.
\end{proof}
\begin{remark} According to the proofs above, the uniform pointwise upper bound \eqref{e:a1-1} of $\mu_x$ is only used to derive
Theorem \ref{p2-1}, see \cite[Page 13, line 7-8]{CKW} for the argument of the assertion that $M(t)\ge 1/2$.
In fact, for this assertion \eqref{e:a1-1} in Assumption ($d$-{\bf Vol}) can be replaced by the condition that \emph{there exists a constant $R_0\ge1$ such that
$$\sup_{x\in B(0,R)}\mu_x\le c_\mu R^{\theta d},\quad \forall R\ge R_0,$$
where $\theta\in (0, {\alpha}/({2d+\alpha}))$ is the constant in Assumption {\bf (HK1)}, and
$c_\mu>0$ is independent of $R_0$ and $R$}. Different from nearest neighbor models, some priori estimates of heat kernel (see e.g.\ \cite{Bar}) are not available, and the maximum principle (see e.g. \cite{BChen}) does not work in our setting. We believe that some kind of upper bounds for $\mu_x$ are required. On the other hand, we also note that, similar to nearest neighbor models (see \cite[(1.5)]{BChen} and \cite[Assumption 1.1(v)]{BKM}),
we need some control on the lower bounds of $\mu_x$, see \eqref{e:a1-2}.
\end{remark}
\subsection{Lower bounds for the heat kernel estimates}\label{section2.3}
Let $Z:=(Z_t)_{t\ge0}=(U_t,$ $X_t)_{t\ge0}$ be the time-space process such
that $U_t=U_0+t$ for any $t\ge0$.
We say that a measurable function $q(t,x)$ on
$[0,\infty)\times V$ is parabolic in an open subset $A$ of
$[0,\infty)\times V$, if for every relatively compact open subset
$A_1$ of $A$, $q(t,x)=\Ee^{(t,x)}q(Z_{\tau_{A_1}})$ for every
$(t,x)\in A_1$.

Let $C_0>0$ be the constant in \eqref{p2-1-1}. For
every $t\ge0$, $R\ge1$ and $x,y\in V$, set $Q(t,x,R)=\left(t, t+ C_0R^{\alpha}\right)\times B(x,R)$.

 \begin{theorem}
 \label{p2-2}
Suppose that Assumptions {\bf($d$-Vol)} and {\bf (HK1)} hold with $\theta \in (0,1)$ and
$R_0\ge1$.
Then, there exist
constants $\delta \in (\theta,1)$ and $R_1\ge1$ such that for all
  $R>R_1$, $x_0 \in B(0,R)$, $R^{\delta}\le r \le R$, $t_0\ge0$ and parabolic function $q$
 on $Q(t_0,x_0, 2r)$,
 \begin{equation}\label{p2-2-1}
 |q(s,x)-q(t,y)|\le C_1\|q\|_{\infty,r}\left(\frac{|t-s|^{1/\alpha}+\rho(x,y)}{r} \right)^\beta,
 \end{equation}
 holds for all $(s,x),(t,y)\in Q(t_0,x_0, r)$ such that
 $(C_0^{-1}|s-t|)^{1/\alpha}+\rho(x,y)\ge 2r^{\delta}$, where
 $\|q\|_{\infty,r}=\sup_{(s,x)\in [t_0, t_0+C_0(2r)^\alpha]\times V}q(s,x),$
and $C_1>0$ and $\beta\in (0,1)$ are constants independent of $R_0$, $R_1$,
$x_0$, $t_0$, $R$, $r$, $s$, $t$, $x$ and $y$.
\end{theorem}
\begin{proof}
According to Theorem \ref{p2-1}, we can follow exactly the same argument of \cite[Theorem 3.8]{CKW}
to obtain the desired assertion. The details are omitted here.
\end{proof}

We also note that, according to \eqref{a2-2}, there exist constants $R_2\ge1$ and $c^*>0$ (independent of $R_0$) such that for every $R>R_2$,
$x\in B(0,6R)$ and $N\ge1$,
\begin{equation}\label{a3-3-1}
\Big\{z\in V: NR\le \rho(z,x)\le c^*NR\Big\}\neq \emptyset.
\end{equation}

\begin{proposition}{\bf(Lower bound for the Dirichlet heat kernel)}
\label{T:lower}
Suppose that Assumptions {\bf($d$-Vol)}, {\bf (HK1)} and {\bf (HK3)} hold with  $\theta\in (0,1)$ and
$R_0\ge1$.
 Then there exist $R_1\ge1$, $\delta \in (\theta,1)$ and $C_1,C_2>0$ $($all three are independent of $R_0$ and $R_1$$)$
 such that for every $R\ge R_1$, $x,y \in B(0,R/4)$ and $R^{\delta\alpha}\le t \le C_1R^{\alpha}$,
\begin{equation}\label{t4-7-1}
p^{B(0,R)}(t,x,y)\ge C_2\Big(t^{-d/\alpha}\wedge \frac{t}{\rho(x,y)^{d+\alpha}}\Big).
\end{equation}
\end{proposition}

\begin{proof}
The proof is split into two steps, and the first one is concerned with near-diagonal lower bound estimates.

{\bf Step (1)}  It follows from \eqref{p2-1-2} that for each $\theta'\in
(\theta,1)$, there exists a constant $\delta \in (\theta',1)$ (here we will take $\delta$ a little bit larger such that
$\delta\in (\theta',1)$) and $R_1\ge1$ such that for all $R\ge
R_1$, $x\in B(0,R)$, $t\ge 2R^{\theta'\alpha}$ and $R^{\delta}\le r
\le R$,
\begin{equation}\label{t4-7-2}
\sum_{y\in B(x,r)^c} p^{B(0,R)}(t/2, x,y)\mu_y\le \Pp_x(\tau_{B(x,r)}\le
t/2)\le  c_1\left(\frac{t}{r^{\alpha}}\right)^{1/3}.
\end{equation}
Choosing $c_0>0$ large enough and $c_2>0$ small enough such that
$c_1c_0^{-\alpha/3}\le 1/3$, $c_1c_2^{1/3}\le 1/3$ and $c_0 c_2^{1/\alpha}\le 1/2$, we have that for all $R>R_1$, $2R^{\theta'\alpha}\le
t \le c_2R^{\alpha}$ and $x\in B(0,R/2)$,
\begin{align*}
&\sum_{y\in B(x, c_0t^{1/\alpha})} p^{B(0,R)}(t/2, x,y)\mu_y\\
&=\sum_{y \in
B(0,R)} p^{B(0,R)}(t/2, x,y)\mu_y\!-\! \sum_{y\in B(x, c_0t^{1/\alpha})^c}
p^{B(0,R)}(t/2,x,y)\mu_y\\
&=1-\Pp_x\big(\tau_{B(0,R)}\le t/2\big)-\sum_{y\in B(x,
c_0t^{1/\alpha})^c}
p^{B(0,R)}(t/2,x,y)\mu_y\\
&\ge 1-\Pp_x\big(\tau_{B(x,R/2)}\le t/2\big)\!-\!\sum_{y\in B(x,
c_0t^{1/\alpha})^c}
p^{B(0,R)}(t/2,x,y)\mu_y\\
&\ge 1-c_1(tR^{-\alpha})^{1/3}-c_1 c_0^{-\alpha/3}\\
&\ge 1-c_1c_2^{1/3}-c_1c_0^{-\alpha/3}=:c_3\ge 1/3,
\end{align*}
where in the first equality we have used the fact that $B(x,c_0t^{1/\alpha})\subseteq
B(0,R)$, and the second inequality follows from \eqref{t4-7-2}.
By the semigroup property and the
Cauchy-Schwarz inequality,
we get that for all $R\ge R_1$, $x\in B(0,R/2)$ and $2R^{\theta'\alpha}\le 2R^{\delta \alpha}\le t \le
c_2R^{\alpha}$,
\begin{equation}\label{t4-7-3}
\begin{split}
p^{B(0,R)}(t,  x,x)&= \sum_{y\in B(0,R)} p^{B(0,R)}(t/2, x,y)^2\mu_y\!\ge \!\!\!\!\sum_{y\in B(x, c_0t^{1/\alpha})}p^{B(0,R)}(t/2, x,y)^2\mu_y\\
&\ge c_4t^{-d/\alpha} \bigg(\sum_{y\in B(x, c_0t^{1/\alpha})}
p^{B(0,R)}(t/2, x,y)\mu_y\bigg)^2 \ge c_5t^{-d/\alpha}.
\end{split}
\end{equation}

On the other hand, under Assumptions {\bf($d$-Vol)}  and {\bf(HK1)}, we have \eqref{p2-2-1}. Let $C_0$ be the constant in \eqref{p2-2-1}, which is used in the definition of $Q(t,x,R)$.  For every fixed $t>0$ and $x \in V$, set
$$f_{t,x}(s,z)=p^{B(0,R)}(t-s,x,z),\quad (s,z)\in [0,t)\times V.$$
It is easy to verify that $f_{t,x}(\cdot,\cdot)$ is parabolic on
$Q\big(0,x,(2^{-1}C_0^{-1}t)^{1/\alpha}\big)$ for every
$x \in B(0,R/2)$
and $2R^{\delta\alpha}\le t \le c_2R^{\alpha}$ with some
$c_2>0$ small enough.
Therefore, according to \eqref{p2-2-1}, there exist $\delta \in (\theta,1)$ and $R_1\ge1$ (for simplicity we adopt
the same $R_1$ and $\delta$ as those in \eqref{t4-7-2}) such that for all
  $R> R_1$, $4R^{\delta\alpha}\le t\le c_2 R^{\alpha}$,
$x\in B(0,R/2)$
and $y \in B(x,3^{-1}(2^{-1}C_0^{-1}t)^{1/\alpha})$ with $\rho(x,y)\ge (2^{-1}C_0^{-1}t)^{\delta/\alpha}$,
\begin{equation}\label{t4-7-4a}
\begin{split}
&\big|p^{B(0,R)}(t, x,y)-p^{B(0,R)}(t,x,x)\big|\\
&=\big|f_{t,x}(0,y)-f_{t,x}(0,x)\big|\le c_6\left(\sup_{s\in B(0,t/2) ,z\in V}|f_{t,x}(s,z)|\right)\left(\frac{\rho(x,y)}{t^{1/\alpha}}\right)^{\beta}\\
&\le c_6\left(\sup_{t\ge 2R^{\delta\alpha} ,z\in V}p^{B(0,R)}(t,x,z)\right)\left(\frac{\rho(x,y)}{t^{1/\alpha}}\right)^{\beta}\le c_7t^{-d/\alpha}\left(\frac{\rho(x,y)}{t^{1/\alpha}}\right)^{\beta},
\end{split}
\end{equation}
where the last inequality is due to \eqref{t4-2-1a}.
Combining this with \eqref{t4-7-3}, we can find constants $c_8>0$ and
$c_9>0$ (small enough satisfying that $c_7(4c_9)^{\beta/\alpha}\le c_5/2$) such that for
$4R^{\delta\alpha}\le t\le c_2 R^{\alpha}$ and $x,y\in B(0,R/2)$ with
$c_8t^{\delta/\alpha}:=(2^{-1}C_0^{-1}t)^{\delta/\alpha}\le \rho(x,y)\le 4c_9t^{1/\alpha}$,
\begin{equation}\label{t4-7-4}
\begin{split}
p^{B(0,R)}(t, x,y)&\ge p^{B(0,R)}(t,x,x)- c_7t^{-d/\alpha}\Big(\frac{\rho(x,y)}{t^{1/\alpha}}\Big)^{\beta}\\
&\ge c_5t^{-d/\alpha}-c_7(2c_9)^{\beta/\alpha}t^{-d/\alpha}\ge
c_{10}t^{-d/\alpha}.
\end{split}
\end{equation}

Next, we consider the case that $x,y\in B(0,R/3)$ with $\rho(x,y)<c_8t^{\delta/\alpha}$. By \eqref{a3-3-1}, we can find $z \in B(0,R/2)$ and $c_{11}>3c_8$ such that $c_{11}t^{\delta/\alpha}\le \rho(z,x)\le c^* c_{11}t^{\delta/\alpha}$. Thus, $2c_8t^{\delta/\alpha}<\rho(y,z)\le
(1+c^*)c_{11}t^{\delta/\alpha}\le c_9t^{1/\alpha}$ (if we choose $R$ large enough). Then, according to the argument of \eqref{t4-7-4a}, we can obtain that
$$
\big|p^{B(0,R)}(t,x,y)-p^{B(0,R)}(t,x,z)\big|
\le c_7t^{-d/\alpha}\Big(\frac{\rho(y,z)}{t^{1/\alpha}}\Big)^{\beta}\le c_{12}t^{-d/\alpha}{t^{-(1-\delta)\beta/\alpha}}.
$$
On the other hand, by \eqref{t4-7-4},
$
p^{B(0,R)}(t,x,z)\ge c_{10}t^{-d/\alpha}.
$
Hence, it holds that
\begin{align*}
p^{B(0,R)}(t, x,y)&\ge p^{B(0,R)}(t,x,z)-\big|p^{B(0,R)}(t,x,y)-p^{B(0,R)}(t,x,z)\big|\\
&\ge
c_{10}t^{-d/\alpha}-c_{12}t^{-d/\alpha}t^{-(1-\delta)\beta/\alpha}\ge
c_{13}t^{-d/\alpha}.
\end{align*}

By now we have proved that for all $R>R_1$, $4R^{\delta\alpha}\le t\le c_2
R^{\alpha}$ and $x,y\in B(0,R/3)$ with
$\rho(x,y)\le 4c_9t^{1/\alpha}$,
$$
p^{B(0,R)}(t, x,y)\ge c_{13}t^{-d/\alpha}.
$$

{\bf Step (2)}  Now, by the strong Markov property, for
all $R>R_1$, $4R^{\delta\alpha}\le t \le c_2R^{\alpha}$, $a \in (0,1)$ small
enough and $x, y \in B(0,R/4)$
with $\rho(x,y)> 4c_9t^{1/\alpha}$,
\begin{equation}\label{t4-7-5}
\begin{split}
&\Pp_x\big(X_{at}\in B(y,2c_9t^{1/\alpha});\tau_{B(0,R)}>at\big)\\
&\ge \Pp_x\bigg(\sigma_{B(y,c_9t^{1/\alpha})}\le a t;
\sup_{s\in[\sigma_{B(y,c_9t^{1/\alpha})}, at]}
\rho(X_s,X_{\sigma_{B(y,c_9t^{1/\alpha})}})< c_9t^{1/\alpha}\bigg)\\
&\ge \Pp_x(\sigma_{B(y,c_9t^{1/\alpha})}\le at)\inf_{z\in B(y,
c_9t^{1/\alpha})}
\Pp_z(\tau_{B(z,c_9t^{1/\alpha})}>at)\\
&\ge \Pp_x(\sigma_{B(y,c_9t^{1/\alpha})}\le at)\inf_{z \in B(0,R)}\Pp_z(\tau_{B(z,c_9t^{1/\alpha})}>at)\\
&\ge \Pp_x(\sigma_{B(y,c_9 t^{1/\alpha})}\!\!\le\!\! at)\left(1\!-\!\frac{c_1a t}{(c_9t^{1/\alpha})^{\alpha}}\right)\!\!\\
&\ge \frac{1}{2}\Pp_x\Big(X_{(at)\wedge
\tau_{B(x,c_9t^{1/\alpha})}}\!\!\in\!\! B(y, c_9t^{1/\alpha})\Big).
\end{split}
\end{equation}
Here the third inequality above is due to
$B(y,c_9t^{1/\alpha})\subseteq B(0,R)$ since $y\in B(0,R/4)$ and
$t\le c_2R^{\alpha}$ for some $c_2$ small enough, and in the fourth
inequality we used \eqref{p2-1-2} thanks to the facts that $at\ge
{(c_9t^{1/\alpha})^{\delta \alpha}}$ and
$c_9t^{1/\alpha}\ge c_94^{1/\alpha}R^{\delta}\ge R^{\theta^{''}}$
for any $\theta^{''}\in (\theta',\delta)$ and $R$ large enough, and
in the last inequality we used the fact that $a$ is small enough such that
$ac_9^{-\alpha}c_{1}\le 1/2$.

Note that $B(x,c_9t^{1/\alpha})\cap B(y,c_9t^{1/\alpha})=\emptyset$ for any $x, y \in B(0,R/4)$ with
$\rho(x,y)\ge 4c_9t^{1/\alpha}$. Then,
by \eqref{e:levy},
for all $R>R_1$, $4R^{\theta'\alpha}\le t \le
c_2R^{\alpha}$ and $x, y \in B(0,R/4)$ with
$\rho(x,y)\ge 4c_9t^{1/\alpha}$, we have
\begin{equation}\label{t4-7-6}
\begin{split}
&\Pp_x\Big(X_{(at)\wedge \tau_{B(x,c_9t^{1/\alpha})}}\in B(y, c_9t^{1/\alpha})\Big)\\
&= \Ee_x\left[ \sum_{s\le{(at)\wedge \tau_{B(x,c_9t^{1/\alpha})}}}\I_{\{X_s\in B(y, c_9t^{1/\alpha})\}}  \right]\\
&\ge\Ee_x\left[\int_0^{{(at)\wedge \tau_{B(x,c_9t^{1/\alpha})}}}
\left(\sum_{u\in B(y, c_9t^{1/\alpha})}
\frac{w_{X_s,u}\mu_u}{\rho(X_s,u)^{d+\alpha}}
\,\right)ds\right]\\
&\ge c_{14}\rho(x,y)^{-d-\alpha}\left(\inf_{v\in B(x,c_9t^{1/\alpha})}\sum_{u \in B(y,c_9t^{1/\alpha})}w_{v,u}\mu_u\right)
 \Ee_x\left[ {{(at)\wedge \tau_{B(x,c_9t^{1/\alpha})}}}\right] \\
&\ge c_{15}\Pp_x\left[{{\tau_{B(x,c_9t^{1/\alpha})}\ge at}}\right]\frac{t^{1+d/\alpha}}
{\rho(x,y)^{d+\alpha}}\ge  \frac{c_{16} t^{1+d/\alpha}}{\rho(x,y)^{d+\alpha}}.
\end{split}
\end{equation}
Here
in the second inequality we have used the fact that for any $v\in
B(x,c_9t^{1/\alpha})$ and $u\in B(y, c_9t^{1/\alpha})$,
$$
\rho(v,u)\le \rho(v, x)+\rho(x,y)+\rho(y,u)\le \rho(x,y)+2c_9t^{1/\alpha}\le 2\rho(x,y),
$$
 in the third inequality we used the fact that
 ${ R^{\theta'}\le
 4^{1/\alpha}R^{\delta}\le t^{1/\alpha}\le c_2^{1/\alpha}R\le R}$ for
 $\theta^{'}\in (\theta,\delta)$ and Assumption {\bf(HK3)},
and the last inequality follows from the same argument in the fourth inequality of
\eqref{t4-7-5}.

Note again that $B(y,c_9t^{1/\alpha})\subseteq B(0,R/3)$ since
$y \in B(0,R/4)$
and $t\le c_2R^{\alpha}$ for some $c_2$ small enough. Since {$4R^{\theta^{'}\alpha}\le 4R^{\delta\alpha}\le t \le
c_2R^{\alpha}\le R^{\alpha}$} for $\theta^{'}\in
(\theta,\delta)$, we can obtain from  \eqref{t4-7-4} that for any $a
\in (0,1/2)$,
$$
\inf_{y\in B(0,R/3),z \in B(y,c_9t^{1/\alpha})}p^{B(0,R)}\big(
(1-a)t,y,z\big)\ge c_{17}t^{-d/\alpha}.
$$
Hence, combining this with \eqref{t4-7-5} and \eqref{t4-7-6}, for every $R>R_1$, $4R^{\delta\alpha}\le t \le c_2R^{\alpha}$ and
$x,y\in B(0,R/4)$ with $\rho(x,y)\ge 4c_9t^{1/\alpha}$,
\begin{align*}
p^{B(0,R)}(t, x,y) &\ge \sum_{z\in B(y, c_9t^{1/\alpha})}
p^{B(0,R)}(at,x,z)
p^{B(0,R)}((1-a)t,z,y)\mu_z\\
&\ge\inf_{z\in B(y, c_9t^{1/\alpha})} p^{B(0,R)}({(1-a)t},z,y)\sum_{z\in B(y,c_9t^{1/\alpha})} p^{B(0,R)}(at,x,z)\mu_z\\
&\ge c_{18}t^{-d/\alpha}\Pp_x\big(X_{at}\in B(y,2c_9t^{1/\alpha}),\tau_{B(0,R)}>at\big)\ge \frac{c_{19} t}{\rho(x,y)^{d+\alpha}}.
\end{align*}

Therefore, by all the estimates above, we prove the desired assertion by change $\delta$ a little large if necessary.
\end{proof}

As a consequence of Proposition \ref{T:lower}, we can obtain the following
lower bound of the heat kernel.

\begin{theorem}{\bf (Lower bound for the heat kernel)}\label{t2-3}
Suppose that {Assumptions {\bf($d$-Vol)},  {\bf (HK1)} and {\bf(HK3)}} hold with $\theta\in (0,1)$ and
$R_0\ge1$.
 Then {there exist $R_1\ge1$ and  $\delta \in (\theta,1)$ $($independent of $R_0$ and $R_1$$)$
such that for every $x,y\in V$ and $t>\big(\rho(0,x)\vee \rho(0,y)\vee R_1\big)^{\delta \alpha}$,
\begin{equation}\label{t2-3-1}
p(t,x,y)\ge C_1\Big(t^{-d/\alpha}\wedge \frac{t}{\rho(x,y)^{d+\alpha}}\Big).
\end{equation} where $C_1>0$ is independent of $R_0$, $R_1$, $t$, $x$ and $y$.
In particular,
there exist $T_0>0$ and $\delta \in (\theta,1)$ $($independent of $R_0$ and $T_0$$)$ such that for all $t\ge T_0$ and $x,y \in B(0,t^{1/(\delta\alpha)})$, \eqref{t2-3-1} holds with $C_1>0$ independent of $R_0$, $T_0$, $t$, $x$ and $y$.}
 \end{theorem}
\begin{proof}
Let $x,y\in V$ and $t>0$ such that $t\ge D(x,y)^{\delta\alpha}$, where $D(x,y)=\rho(0,x)\vee \rho(0,y)\vee R_1$ for some large $R_1\ge 1$ and $\delta\in (\theta,1)$. Then, taking $R\ge D(x,y)$ such that
$x,y\in B(0,R)$ and $R^{\delta\alpha}\le t \le C_1R^{\alpha}$ with $C_1$ being the constant in
Proposition \ref{T:lower}, and applying \eqref{t4-7-1} with $t$ and $R$ above, we can get
$$
p(t,x,y)\ge p^{B(0,R)}(t,x,y)\ge c_1\Big(t^{-d/\alpha}\wedge \frac{t}{\rho(x,y)^{d+\alpha}}\Big).
$$ This proves \eqref{t2-3-1} by choosing $\delta$ a little bit larger.
Choosing $R=t^{1/(\delta\alpha)}$ and then renaming $R_1^{\delta\alpha}$ as $T_0$ in the estimate above, we can prove the second desired assertion also by taking $\delta$ a little bit larger.
\end{proof}

\section{Green Function Estimates and Elliptic Harnack Inequalities}\label{sec4}
In this section, we give some estimates of Green functions and
then give a consequence of elliptic Harnack inequalities.
The results in this section is used in Subsection \ref{EHI-cou}.
Throughout this section, we assume that {\it $(G,\mu)$ satisfies the global $d$-set condition; that is,
\begin{equation}\label{e4-1}
c_{\mu}^{-1}\le \mu_x\le c_\mu,\ \ \ \ c_\mu^{-1}r^d \le \mu\big(B(x,r)\big)\le c_\mu r^d,\quad r\ge 1,\ x\in V.
\end{equation}}
We also suppose in this section
that {\it $w_{x,y}>0$ for all $x,y\in V$ with $x\neq y$.}

\subsection{Upper bounds for the Dirichlet heat kernel: small times}
In this subsection we present some upper bounds of
the Dirichlet heat kernel for small times,
which are used to obtain estimates for Green functions in the next subsection.

In the nearest neighbor conductance models,
some priori estimates are used for the small time heat kernel estimates,
see \cite[Corollaries 11 and 12]{D} or \cite[Theorems 2.1 and 2.2]{Fo}. However, such estimates are not known for the long range case.
In order to overcome this obstacle, we will adopt the localization approach (see \cite[Section 2.2]{CKW} for more details).
We note that for our arguments to work, we need the
global $d$-set condition \eqref{e4-1} and the assumption that
$w_{x,y}>0$ for all $x,y\in V$ with $x\neq y$.
Since finally we
will apply results in this part to
Subsection \ref{EHI-cou}, in which we show
the elliptic Harnack inequalities do not hold in general on long-range random conductance models,
we are satisfied with the assumptions. It is an interesting open problem how much we can relax the assumptions.

For any fixed $R\ge1$, $\gamma\in (0,1)$ and $x_0\in B(0,2R)$, we define the following
symmetric regular Dirichlet form $(\hat D^{x_0, \gamma,R},\hat
\F^{x_0, \gamma,R})$:
\begin{align*}
\hat D^{x_0,\gamma,R}(f,f)=&\sum_{x,y\in
V:\rho(x,y)\le \gamma R}\big(f(x)-f(y)\big)^2\frac{\hat w_{x,y}}{\rho(x,y)^{d+\alpha}}\mu_x\mu_y,\quad
f\in
\hat \F^{x_0,\gamma,R},\\
\hat \F^{x_0,\gamma,R}=&\{f \in L^2(V;\mu): \hat D^{x_0,\gamma,
R}(f,f)<\infty\},
\end{align*}
where
\begin{equation}\label{e4-2a}
\hat w_{x,y}=
\begin{cases}
& w_{x,y},\ \ \ \text{if}\ x\in B(x_0, R)\ \text{or}\ y\in B(x_0, R),\\
& \,\, 1,\ \ \ \ \ \ \text{otherwise}.
\end{cases}
\end{equation}
Here we omit the parameters $x_0$ and $R$ in the definition of $\hat w_{x,y}$ for simplicity.
Denote by  $(\hat X_t^{x_0,\gamma,R})_{t\ge 0}$  the symmetric Hunt process
associated with $(\hat D^{x_0, \gamma,R},\hat
\F^{x_0, \gamma,R})$.

For every $r>0$, set
$$
\Theta(r):=1+ \sup_{x,y\in B(0,r)}w^{-1}_{x,y}.
$$

\begin{lemma}\label{l2-1}
Suppose that Assumption {\bf (HK1)} holds with $\theta\in (0,1)$ and $R_0\ge1$. Then, for any
$\gamma\in (0,\alpha/(3(d+\alpha))]$,
there exists a constant $R_1\ge1$ such
that for every $R>R_1$, $x_0\in B(0,2R)$, $R^{\theta}\le r\le R$, $x\in V$ and $t>0$,
\begin{equation}\label{l2-1-1}
\Pp_x\Big(\rho\big(\hat X_t^{x_0,\gamma, r},x\big)>r\Big)\le C_1\Theta(4R)^{d/\alpha}\frac{t}{r^\alpha},
\end{equation}
where $C_0,C_1>0$ are independent of $R_0$, $R_1$, $R$, $r$, $x_0$, $x$ and $t$ $($but may depend on $\gamma$$)$.
\end{lemma}
\begin{proof} We only need to verify \eqref{l2-1-1} the case that $0<t\le (\gamma r)^\alpha/2$.  The proof is split into two steps.

{\bf Step (1)}
By Assumption {\bf (HK1)} and the definition of $\hat w_{x,y}$, we can verify that for every $1\le r\le R$
and $x_0\in B(0,2R)$,
\begin{equation}\label{l2-1-2}
\sup_{x \in V}\sum_{y\in V:\rho(x,y)\le
c_*r}\hat w_{x,y}^{-1}\mu_y\le c_0\Theta(4R)r^{d}
\end{equation} with $c_*:=8c_\mu^{2/d}$,
and also that there is a constant $R_0\ge 1$ such that for all $R\ge R_0$, $x_0\in B(0,2R)$ and $R^{\theta}\le r \le R$,
\begin{equation}\label{l2-1-1a}
 \sup_{x\in V}\sum_{y\in V:\rho(x,y)\le  \gamma r}
\frac{\hat w_{x,y}\mu_y}{\rho(x,y)^{d+\alpha-2}}\le c_0(\gamma r)^{2-\alpha}.
 \end{equation}

Using \eqref{l2-1-2} and  \eqref{e4-1},
 and following the argument of \eqref{t4-2-2}, we can obtain that for any $x_0\in
B(0,2R)$, $1\le r\le R$, $x\in V$ and any measurable function $f$ on $V$,
\begin{equation}\label{t4-6-1}
\begin{split}&\sum_{z\in B(x,r)}(f(z)\!-\!(f)_{B(x,r)})^2\mu_z\le c_1\Theta(4R) r^{\alpha}\!\!\!
\sum_{z,y\in B(x,r)}\!\!(f(z)\!-\!f(y))^2 \frac{\hat w_{z,y}}{\rho(z,y)^{d+\alpha}}\mu_y\mu_z.\end{split}
\end{equation}
Let $\hat p^{x_0,\gamma,r}(t,x,y)$ be the heat kernel associated with the process $(\hat X_t^{x_0,\gamma,r})_{t\ge 0}$, and let
$\psi_{x_0}(t,x)=\hat p^{x_0,\gamma,r}(2t,x,x)=\sum_{z\in V}\hat p^{x_0,\gamma,r}(t,x,z)^2\mu_z$ for any $t>0$ and $x\in V$.
Then, it follows from \eqref{t4-6-1} and the argument of \eqref{t4-2-3} that for all $x_0\in B(0,2R)$, $x\in V$, $1\le r \le R$, $0<t\le (\gamma r)^{\alpha}/2$
and $1\le r(t)\le \gamma r$,
\begin{equation*}
\psi_{x_0}'(t,x)\le c_2\Theta(4R)^{-1}r(t)^{-\alpha}\big(\psi_{x_0}(t,x)-c_3r(t)^{-d}\big).
\end{equation*}
Here and in what follows constants $c_i$ may depend on $\gamma$.
Thus,  further following  part (3) in the proof of Proposition \ref{p-on}, we  can finally obtain for all $x_0\in B(0,2R)$ and $1\le r \le R$,
\begin{equation}\label{l2-1-4}
\hat p^{x_0,\gamma,r}(t,x,y)\le c_4\Theta(4R)^{d/\alpha}t^{-d/\alpha},\quad x,y\in V,
0<t\le (\gamma r)^{\alpha}/2.
\end{equation}

{\bf Step (2)} For simplicity, in the following we omit the index $x_0$.
According to \eqref{l2-1-4} and \cite[Theorem (3.25)]{CKS},
 for any $x,y\in V$, $1\le r\le R$ and $0<t\le {(\gamma r)^\alpha/2}$, we have
$$\hat p^{x_0,\gamma, r}(t,x,y)\le c_4^*
\Theta(4R)^{d/\alpha}t^{-d/\alpha}
\inf_{\psi\in L^\infty(V;\mu)}\exp(\psi(x)-\psi(y)+c(\psi)
t),$$ where
$c(\psi)=\sup_{x\in V} b(\psi,x)$ and $$b(\psi,x)=
\sum_{z\in V:\rho(z,x)\le \gamma r}\frac{\hat w_{x,z}\mu_z}{\rho(x,z)^{d+\alpha}}
(e^{\psi(z)-\psi(x)}-1)^2.$$

Next, for fixed $x,y \in V$ and $\lambda>0$, define
$$\psi_\lambda(z)=\lambda (\rho(x,y)\wedge \rho(z,x))\in L^\infty(V;\mu).$$
Then,
\begin{align*}
b(\psi_\lambda,x)&\le
\sum_{z\in V: \rho(x,z)\le \gamma r}\frac{\hat w_{x,z}\mu_z}{\rho(x,z)^{d+\alpha}}
(e^{\lambda (\rho(x,z)\wedge \rho(x,y))}-1)^2\\
&\le
\lambda^2
e^{2\lambda\gamma r}
\sum_{z\in V: \rho(z,x)\le \gamma r}\hat w_{x,z}\mu_z\rho(x,z)^{2-d-\alpha}=:J(x),
\end{align*}
where in the first inequality we used the facts that
$s\mapsto (e^s-1)^2$
is increasing for $s\ge0$ and
$|\psi_\lambda(z)-\psi_\lambda(y)|\le \lambda\big(\rho(y,z)\wedge \rho(x,y)\big)$ for any $x,y,z\in V$, and the second inequality follows from the inequality that
$|e^s-1|^2\le s^2 e^{2|s|}$
for $s\ge0$.
On the other hand, it is obvious that, by \eqref{l2-1-1a}, we can find a constant $R_1\ge1$ such that for all $R>R_1$, $x_0\in B(0,2R)$ and $R^{\theta}\le r \le R$,
$$
\sup_{x \in V}J(x)\le c_5
e^{2\lambda\gamma r}(\lambda \gamma r)^2
(\gamma r)^{-\alpha}.
$$
Combining both estimates above, we arrive at that for all $R>R_1$, $x_0\in B(0,2R)$ and $R^{\theta}\le r \le R$,
\begin{equation*}
\sup_{x \in V}b(\psi_\lambda,x)\le c_5
e^{3\lambda \gamma r}(\gamma r)^{-\alpha},
\end{equation*}
where we have used the fact that
$s^2e^{2s}\le 2e^{3s}$
for $s\ge0$.

Now we suppose that $R>R_1$, $x_0\in B(0,2R)$ and $R^{\theta}\le r \le R$.
Hence, for any $0<t\le {(\gamma r)^\alpha/2}$ and $x,y \in V$,
$$\hat p^{\gamma,r}(t,x,y)\le
c^*_4\Theta(4R)^{d/\alpha}t^{-d/\alpha}\exp\left(-\lambda \rho(x,y) +c_5(\gamma r)^{-\alpha}
e^{3\lambda \gamma r} t\right).$$
Set
$$\lambda=\frac{1}{3\gamma r}
\log\Big(\frac{(\gamma r)^\alpha}{t}\Big)$$
in the right side of the inequality above, we get for all $0<t\le {(\gamma r)^\alpha/2}$ and $x,y \in V$,
$$\hat p^{\gamma,r}(t,x,y)\le c_6\Theta(4R)^{d/\alpha}
t^{-d/\alpha}\Big(\frac{t}
{(\gamma r)^\alpha}\Big)^{{\rho(x,y)}/{({3\gamma r})}}.$$

Therefore, for all $x\in V$ and $0<t\le (\gamma r)^{\alpha}/2$,
\begin{align*}
\Pp_x\Big(\rho(\hat X_t^{x_0,\gamma,r},x)>r\Big)&=\sum_{y\in V:\rho(x,y)>r}\hat p^{\gamma,r}(t,x,y)\mu_y\\
&\le c_7\!\Theta(4R)^{d/\alpha}t^{-d/\alpha}\!\!\sum_{k=0}^{\infty}
\sum_{y\in V: 2^kr<\rho(x,y)\le 2^{k+1}r}\!\!\!\!\!{\Big(\frac{t}{(\gamma r)^\alpha}
\Big)^{{\rho(x,y)}/({3\gamma r})}}\\
&\le c_8\Theta(4R)^{d/\alpha}t^{-d/\alpha}\sum_{k=0}^{\infty}
\mu\big(B(x,2^{k+1}r)\big)\Big(\frac{t}{(\gamma r)^\alpha}
\Big)^{{2^{k}r}/({3\gamma r})}\\
&\le c_{9}\Theta(4R)^{d/\alpha}t^{-d/\alpha}
\sum_{k=0}^{\infty}2^{(k+1)d}r^d\Big(\frac{t}{(\gamma r)^\alpha}
\Big)^{{2^{k}}/{(3\gamma) }}\\
&\le c_{10}\Theta(4R)^{d/\alpha}\Big(\frac{t}{(\gamma r)^\alpha}
\Big)^{({1}/({3\gamma}))-({d}/{\alpha})}.
\end{align*}
Taking {$\gamma\le
\alpha/(3(d+\alpha))$} (in particular, $({1}/({3\gamma}))
-({d}/{\alpha})\ge 1$) in the last inequality immediately yields the desired assertion \eqref{l2-1-1}.
\end{proof}

By Lemma \ref{l2-1}, we can further establish the following estimate for exit time of the process $X$.
\begin{lemma}\label{l2-2}
Suppose that Assumption {\bf (HK1)} holds with some constant $\theta\in (0,1)$ and $R_0\ge1$. Then for every
{$\theta'\in (\theta,1)$},
 there exist constants $R_1\ge 1$ and $C_0,C_1>0$ $($which
are independent of $R_0$ and $R_1$$)$ such
that for every $R>R_1$, $x\in B(0,2R)$, $R^{\theta'}\le r\le 2R$ and $t>0$,
\begin{equation}\label{l2-2-1b}
\Pp_{x}(\tau_{B(x,r)}\le t)\le C_1\Theta(4R)^{d/\alpha}\frac{t}{r^\alpha}.
\end{equation}
\end{lemma}
\begin{proof}
It suffices to verify the desired assertion for the case that {$0<t\le r^{\alpha}/2$}. The proof is split into two steps.

{\bf Step (1)} We suppose that $\gamma\in (0,\alpha/(3(d+\alpha))]$, $R\ge R_1$ for some $R_1$ large enough, $x\in B(0,2R)$, $R^{\theta}\le r/2 \le R$ and
$0<t\le r^{\alpha}/2$. According to \eqref{l2-1-1},
\begin{equation*}
\sup_{s\in [t,2t]}\sup_{y \in V}\Pp_y\Big( \rho\big(\hat X_{s}^{x,\gamma,r},y\big)>\frac{r}{2} \Big)\le
c_1\Theta(4R)^{d/\alpha}\frac{t}{r^{\alpha}}.
\end{equation*}
Hence, by the strong Markov property,
\begin{align*}
\Pp_{x}\big(\hat\tau_{B(x,r)}^{x,\gamma, r}\le t\big)
&\le  \Pp_{x}\Big(\hat\tau_{B(x,r)}^{x,\gamma,r}\le t,
\rho\big(\hat X_{2t}^{x,\gamma, r},x\big)\le \frac{r}{2} \Big)
+\Pp_{x}\Big(\rho\big(\hat X_{2t}^{x,\gamma, r},x\big)> \frac{r}{2} \Big)\\
&\le \Ee_{x}\left[\I_{\{\hat \tau^{x,\gamma, r}_{B(x,r)}\le t\}}
\Pp_{\hat X_{\hat\tau_{B(x,r)}^{
x,\gamma,r}}^{x,\gamma,r}} \Big(\rho\big
(\hat X_{2t-\hat\tau^{x,\gamma,
r}_{B(x,r)}}^{x, \gamma, r},\hat X_0^{
x,\gamma,r}\big)>\frac{r}{2}\Big)\right]+c_1\Theta(4R)^{d/\alpha}\frac{t}{r^{\alpha}}\\
&\le \sup_{y \in V}\sup_{s \in [t,2t]}\Pp_{y}\Big(\rho\big(\hat
X_{s}^{x,\gamma,r},y\big)
>\frac{r}{2}\Big)+c_1\Theta(4R)^{d/\alpha}\frac{t}{r^{\alpha}}\\
&\le
2c_1\Theta(4R)^{d/\alpha}\frac{t}{r^{\alpha}},
\end{align*}
where $\hat \tau_A^{x,\gamma,r}:=\inf\{t>0; \hat X_t^{x,\gamma,r}\notin A\}$ is the first exit time from $A\subset V$ for
the process $(\hat X_t^{x,\gamma,r})_{t\ge 0}$.

{\bf Step (2)} Similar to the argument in \cite[Section 3.1]{CKW}, we define the following Dirichlet form $(\hat D^{x, r},\hat
\F^{x, r})$
\begin{align*}
\hat D^{x,r}(f,f)=&\sum_{y,z\in
V}\big(f(y)-f(z)\big)^2\frac{\hat w_{y,z}}{\rho(y,z)^{d+\alpha}}\mu_y\mu_z,\quad
f\in
\hat \F^{x,r},\\
\hat \F^{x,r}=&\{f \in L^2(V;\mu): \hat D^{x,
r}(f,f)<\infty\},
\end{align*}
where $\hat w_{x,y}$ is defined by \eqref{e4-2a} with $x_0$ and $R$ replaced by $x$ and $r$ respectively.
In particular, by the definition of $(\hat D^{x, r},\hat
\F^{x, r})$,  we have immediately
$$
\Pp_{x}\big(\tau_{B(x,r)}\le
t\big)=\Pp_{x}\big(\hat\tau^{x,r}_{B(x,r)}\le
t\big), \quad x\in V, r,t>0,
$$
where $(\hat X_t^{x,r})_{t\ge0}$ is the Hunt process associated with $(\hat D^{x, r},\hat
\F^{x, r})$, and $\hat \tau_A^{x,r}:=\inf\{t>0: \hat X_t^{x,r}\notin A\}$ for a subset $A \subset V$.
Furthermore, according to \cite[Lemma 3.1]{CKW}, we find that \begin{align*}
& \Pp_x\big(\hat\tau^{x,r}_{B(x,r)}\le
t\big)\le \Pp_x\big(\hat\tau^{x,\gamma,r}_{B(x,r)}\le
t\big)+c_2t\sup_{y \in B(x,r)}\sum_{z\in V:
\rho(y,z)>\gamma r}\frac{ w_{y,z}\mu_z}{\rho(y,z)^{d+\alpha}}
.
\end{align*}
Noting that $\gamma r>\gamma R^{\theta'}\ge R^{\theta}$ for $R$ large enough, Assumption {\bf(HK1)} implies
$$
\sup_{y \in B(x,r)}\sum_{z\in V:
\rho(y,z)>\gamma r}\frac{ w_{y,z}\mu_z}
{\rho(y,z)^{d+\alpha}}
\le c_3r^{-\alpha}.
$$

Combining all the estimates above yields the desired conclusion
\eqref{l2-2-1b}.
\end{proof}

Now, we are in the position to present the upper bound of the Dirichlet heat kernel
for small times.

 \begin{proposition}\label{T:small}
Suppose that Assumptions {\bf (HK1)} and {\bf (HK2)} hold with constants $\theta\in (0,1)$ and $R_0\ge1$.
Then, for each $\theta'\in (\theta,1)$, there is a constant
$R_1\ge1$ such that for all
$R>R_1$, $x,y \in B(0,R)$ with $\rho(x,y)\ge R^{\theta'}$, and $0<t\le R^\alpha$,
\begin{equation}\label{t4-6-0}
p^{B(0,R)}(t,x,y)\le C_1\Theta(4R)^{k}
\left(t^{-d/\alpha}\wedge\frac{t}{\rho(x,y)^{d+\alpha}}\right),
\end{equation}
where $C_0,C_1>0$ and $k\ge d/\alpha$ are independent of $R_0$, $R_1$, $R$, $x$, $y$ and $t$.
\end{proposition}
\begin{proof} We only sketch the proof. First, we note that \eqref{t4-6-1} holds with $\hat w_{z,y}$ replaced by
$w_{z,y}$ and for all $x\in B(0,R)$.
Following the proof of Proposition \ref{p-on} and carefully tracking the dependence on $\Theta(4R)$, we can obtain that
there are constants
$R_1\ge1$ such that for all
$R>R_1$, $x,y \in B(0,R)$ and $0<t\le R^\alpha$,
\begin{equation}\label{t4-6-1a}
p^{B(0,R)}(t,x,y)\le c_2\Theta(4R)^{d/\alpha}t^{-d/\alpha}.
\end{equation}
Second, using \eqref{t4-6-1a}, \eqref{l2-2-1b} (which requires that $x,y\in B(0,R)$ with $\rho(x,y)\ge R^{\theta'}$) and Assumption {\bf (HK2)}, and repeating the argument of Proposition \ref{T:ofde}, we can  obtain the desired conclusion.
(We note that in the proof of Proposition \ref{T:ofde}, we apply the induction procedure. Hence
the term $\Theta(4R)$ will be multiplied for several times.)
\end{proof}

\subsection{Estimates for Green functions}
The aim of this part is to
obtain some estimates for Green functions.
\begin{proposition}\label{l5-1}
Let $d>\alpha$. Suppose that Assumptions {\bf (HK1)}, {\bf (HK2)} and {\bf (HK3)} hold with constants $\theta\in (0,1)$
and $R_0\ge1$.
Then, for every $\theta'\in (\theta,1)$, there
exist constants $R_1\ge 1$ and $k\ge d/\alpha$
such that the following hold for every $R\ge R_1$,
\begin{itemize}
\item[(1)] for any $x,y\in B(0,R)$ with $\rho(x,y)\ge R^{\theta'}$,
\begin{equation}\label{l5-1-1}
G^{B(0,R)}(x,y)\le C_1\Theta(4R)^{k}\rho(x,y)^{-d+\alpha};
\end{equation}
\item[(2)] for $x,y\in B(0,R/4)$,
\begin{equation}\label{l5-1-2} G^{B(0,R)}(x,y)\ge C_2\Big(\rho(x,y)^{-d+\alpha}\wedge R^{-\theta'(d-\alpha)}\Big),\end{equation}
\end{itemize}
where $C_0,C_1,C_2,k$ are positive constants independent of $R_1$, $R$, $x$ and $y$.
\end{proposition}
\begin{proof}
According to \eqref{t4-2-1a} and \eqref{t4-6-0}, under Assumptions {\bf(HK1)} and {\bf (HK2)} there exist constants
$R_1\ge1$ and $k\ge d/\alpha$
such that for all $R\ge R_1$ and $x,y\in B(0,R)$ with $\rho(x,y)\ge R^{\theta'}$,
\begin{equation}\label{l5-1-6}
p^{B(0,R)}(t,x,y)\le
\begin{cases}
{c_1\Theta(4R)^{k}t}{\rho(x,y)^{-d-\alpha}},\quad  & 0<t\le \rho(x,y)^{\alpha},\\
 c_1t^{-d/\alpha},\quad  & t>\rho(x,y)^{\alpha}\ge R^{\theta'\alpha},
\end{cases}
\end{equation}
which implies that
\begin{align*}
G^{B(0,R)}(x,y)&\le c_1\left[\int_0^{\rho(x,y)^{\alpha}}\frac{\Theta(4R)^{k}t}{\rho(x,y)^{d+\alpha}}\,dt+
\int_{\rho(x,y)^{\alpha}}^{\infty}t^{-d/\alpha}\,dt\right]\le c_2\Theta(4R)^{k}\rho(x,y)^{-d+\alpha},
\end{align*} where in the last inequality we used $d>\alpha$. This proves \eqref{l5-1-1}.

On the other hand, by \eqref{t4-7-1}, under Assumptions {\bf(HK1)} and {\bf (HK3)} there exists a constant
$R_2\ge1$ such that for all $R\ge R_2$, $x,y\in B(0,R/4)$ and $R^{\theta'\alpha}\le t \le c_3R^{\alpha}$,
$$
p^{B(0,R)}(t,x,y)\ge c_4\Big(t^{-d/\alpha}\wedge \frac{t}{\rho(x,y)^{d+\alpha}}\Big).
$$
If $x,y\in B(0,R/4)$ such that $\rho(x,y)\le R^{\theta'}$, then
$\rho(x,y)\le t^{1/\alpha}$ for any $R^{\theta'\alpha}\le t \le c_3R^{\alpha}$, and so
\begin{align*}
& G^{B(0,R)}(x,y)\ge \int_{R^{\theta'\alpha}}^{c_3R^{\alpha}}
p^{B(0,R)}(t,x,y)\,dt \ge c_4\int_{R^{\theta'\alpha}}^{c_3R^{\alpha}}t^{-d/\alpha}\,dt
\ge c_5R^{-\theta'(d-\alpha)}.
\end{align*}
While for any $x,y\in B(0,R/4)$ with $\rho(x,y)>R^{\theta'}$, it holds
\begin{align*}
G^{B(0,R)}(x,y)&\ge c_4\Big(\int_{R^{\theta'\alpha}}
^{\rho(x,y)^{\alpha}\wedge (c_3R^{\alpha})}\frac{t}{\rho(x,y)^{d+\alpha}}\,dt+
\int_{\rho(x,y)^{\alpha}\wedge (c_3R^{\alpha})}^{c_3R^{\alpha}}t^{-d/\alpha}\,dt\Big)\\
&\ge c_6\rho(x,y)^{-d+\alpha},
\end{align*}
where in the last inequality we have used the facts that $d>\alpha$ and $\rho(x,y)\le {R}/{2}$ for any
$x,y \in B(0,R/4)$. Combining both estimates above, we obtain \eqref{l5-1-2}.
\end{proof}

According to Proposition \ref{l5-1} and its proof, we can also obtain the following estimates for global Green functions.

\begin{proposition} Assume that {Assumptions {\bf{(HK1)}}, {\bf (HK2)} and {\bf(HK3)} hold with $\theta\in (0,
 1)$ and $R_0\ge1$}. We further suppose that
$\sup_{r\ge 1}\Theta(r)<\infty$.
Then, for every {$\theta'\in (\theta,1)$}, there exists a constant
$R_1\ge1$ such that for all
$x,y\in V$ with $$\rho(x,y)\ge {4(R_1\vee \rho(0,x)\vee \rho(0,y))^{\theta'}},$$ it holds
\begin{equation}\label{l5-1-4}
C_1\rho(x,y)^{-d+\alpha}\le G(x,y)\le C_2\rho(x,y)^{-d+\alpha},
\end{equation}
where $C_1,C_2>0$ are independent of $R_0$, $R_1$, $x$ and $y$.
\end{proposition}
\begin{proof} For any $x,y\in V$,  define $D(x,y):=R_1\vee \rho(0,x)\vee \rho(0,y)$
 {with $R_1\ge1$ being the constant in Proposition \ref{l5-1}.} Suppose that $x,y\in V$ satisfy
 $\rho(x,y)\ge 4D(x,y)^{\theta'}$.
Then, applying  (the first inequality in) \eqref{l5-1-6} and \eqref{l2-2-1b}, we can follow the proof of Proposition \ref{T:ofe} with $N(t,x,y):=\rho(x,y)^{{(d+\alpha)}/{\alpha}}$ and get that for any $R$ large enough and any $0<t\le \rho(x,y)^{\alpha}$,
\begin{align*}
p(t,x,y)&\le c_1[\Theta(
\rho(x,y)^{{(d+\alpha)}/{\alpha}})]^{k}
\frac{t}{\rho(x,y)^{d+\alpha}}\le
\frac{ c_2t}{\rho(x,y)^{d+\alpha}}.
\end{align*}
This along with \eqref{t4-2-1} yields that
\begin{align*}
G(x,y)&\le c_3\Bigg(\int_0^{\rho(x,y)^{\alpha}}
  \frac{t}{\rho(x,y)^{d+\alpha}}\,dt+\int_{\rho(x,y)^{\alpha}}^{\infty}t^{-d/\alpha}\,dt\Bigg)\le c_{4}\rho(x,y)^{-d+\alpha},
\end{align*}
proving the desired upper bound in \eqref{l5-1-4}.

On the other hand, the desired lower bound in \eqref{l5-1-4} follows from  \eqref{l5-1-2} by taking $R=4D(x,y)$. The proof is complete.
\end{proof}

\subsection{Consequence of elliptic Harnack inequalities}
Let ${\mathcal L}$
be the generator associated with the process $X$. For any subset $A\subseteq V$, we say that $u:A \rightarrow \R$ is harmonic
with respect to ${\mathcal L}$ on $A$,  if for any $x\in A$, ${\mathcal L}u(x)=0.$ We call that the elliptic Harnack inequality $(\EHI)$ (associated with ${\mathcal L}$) holds at $x_0\in V$, if
there are constants $R_0:=R_0(x_0)\ge 1$ and $c_1:=c_1(x_0,R_0)\ge 1$ such that for every $R\ge R_0$ and
every non-negative function $f(x)$ on $V$ which is harmonic on $B(x_0,2R)$, it holds that
$$\sup_{x\in B(x_0,R)} f(x)\le c_1(x_0,R_0) \inf_{x\in B(x_0,R)} f(x).$$  We emphasize that here we use a weaker version of \EHI, where the constant $c_1(x_0,R_0)
$ may depend on $x_0$ and $R_0$.
Note that
unlike \cite{CKW3}, in the present setting the associated elliptic Harnack inequality is not necessarily translation or scaling invariant.

\begin{proposition}\label{p5-1}
Let $d>\alpha$. Suppose that Assumptions  {\bf (HK1)}, {\bf (HK2)} and {\bf(HK3)} hold with  some
constants $\theta\in (0,1)$ and $R_0\ge1$.
Assume further that $\sup_{r\ge1}\Theta(r)<\infty.$
If $\EHI$ holds at $x_0\in V$, then for every $\theta'\in (\theta,1)$, there exist constants $R_1:=R_1(x_0)\ge1$ and
$c_0:=c_0(x_0,R_1)>0$
such that for all $R\ge R_1$ and $z\in B(x_0,4R)^c$,
\begin{equation}\label{p5-1-1}
w_{x_0,z}\le c_0\left(\sup_{v\in B(x_0,2R),v\neq x_0}w_{v,z}\right)R^{\alpha+\theta'(d-\alpha)}.
\end{equation}
 \end{proposition}
\begin{proof}
Without loss of generality, we will assume that $x_0=0$.
For any $R\ge1$ and $z \in B(0,4R)^c$,
define $$f_z(x)=\Pp_x(X_{\tau_{B(0,2R)}}=z),\quad x\in B(0,2R),$$
which is a harmonic function on $B(0,R)$.
By $\EHI$,
there are constants $R_0\ge 1$ and $c_1:=c_1(x_0,R_0)\ge 1$ such that
for all $z \in B(0,4R)^c$,
\begin{equation}\label{p5-1-1a}
f_z(x)\le c_1 f_z(y),\quad x,y\in B(0,R).
\end{equation}
Note that, according to the Ikeda-Watanabe formula (see \cite[Theorem 2]{IW}),
it holds that
$$f_z(x)=\sum_{v\in B(0,2R)}G^{B(0,2R)}(x,v)\frac{w_{v,z}\mu_v}{\rho(v,z)^{d+\alpha}}
,\quad x\in B(0,R).$$
According to Theorem \ref{p2-1} (i.e., \cite[Theorem 3.4]{CKW}), under Assumption
{\bf (HK1)} we have \eqref{l2-2-1}. In particular,  there exists a constant $R_0\ge1$
(here for notational simplicity, we take the same $R_0$ as above)
such that
for all $R\ge R_0$ and $x \in B(0,R)$,
\begin{equation}\label{p5-1-2}
\begin{split}
&\sum_{y\in B(0,2R)}G^{B(0,2R)}(x,y)\mu_y= \Ee_x[\tau_{B(0,2R)}]
\le \Ee_x[\tau_{B(x,4R)}]\le c_2R^{\alpha},
\end{split}
\end{equation}
where $c_2>0$ is independent of
$R_0$ and $R$.
(We note that here $c_2$ is also independent of $x_0$ if we consider that $\EHI$ holds at $x_0\in V$ in the beginning.)

Since ${\rho(z,0)}/{2}\le \rho(z,v)\le{3\rho(z,0)}/{2}$ for all $z \in B(0,4R)^c$ and
$v \in B(0,R)$, for any $R\ge R_0$ with
$R/2\ge R^{\theta'},$ $y\in B(0,R)$ with $\rho(y,0)\ge R/2$ and $z \in B(0,4R)^c$,
\begin{align*}
f_z(y)&=\sum_{v\in B(0,2R)}G^{B(0,2R)}(y,v)\frac{w_{v,z}\mu_v}{\rho(v,z)^{d+\alpha}}
\\
&\le c_3\rho(z,0)^{-d-\alpha}\\
&\quad \times
\Bigg[\Big(\sup_{v\in B(0,2R),v\neq 0}
w_{v,z}\Big)\cdot\Big(\sum_{v\in B(0,2R)}G^{B(0,2R)}(y,v)\mu_v\Big)+w_{0,z}G^{B(0,R)}(y,0)\Bigg]\\
&\le c_4\rho(z,0)^{-d-\alpha}\Big(R^{\alpha}\cdot \sup_{v\in B(0,2R),v\neq 0}
w_{v,z}+w_{0,z}R^{-d+\alpha}\Big),
\end{align*}
where $c_4>0$ is independent of $R_0$ and $R$,
and in the last inequality we have used \eqref{p5-1-2} and \eqref{l5-1-1}, thanks to the fact that $\rho(y,0)\ge R^{\theta'}$.

On the other hand, by \eqref{l5-1-2}, it holds that
\begin{align*}
f_z(0)&=\sum_{v\in B(0,2R)}G^{B(0,2R)}(0,v)
\frac{w_{v,z}\mu_v}{ \rho(v,z)^{d+\alpha}}
\\
&\ge c_5\rho(z,0)^{-d-\alpha}G^{B(0,2R)}(0,0)w_{0,z}\ge c_6\rho(z,0)^{-d-\alpha}w_{0,z}R^{-\theta'(d-\alpha)},
\end{align*}
where $c_6>0$ is independent of $R_0$ and $R$.

Combining both estimates above with \eqref{p5-1-1a}, we find that for all $R\ge R_0$ and $z \in B(0,4R)^c$
\begin{align*}
w_{0,z}R^{-\theta'(d-\alpha)}\le
c_7w_{0,z}R^{-d+\alpha}+
c_7\left(\sup_{v\in B(0,2R),v\neq 0}
w_{v,z}\right)R^{\alpha}.
\end{align*}
Noting that
$c_7$ is independent of $R$ (but may depend on $R_0$)
and that $\theta'<1$, we can take
$R_1$ large enough such that
$c_7R^{-d+\alpha}\le R^{-\theta'(d-\alpha)}/2$ for
all $R\ge R_1\ge R_0$. Therefore, for all $R\ge R_1$,
\begin{equation*}
w_{0,z}R^{-\theta'(d-\alpha)}\le 2 c_7
\left(\sup_{v\in B(0,2R),v\neq 0}
w_{v,z}\right)R^{\alpha},
\end{equation*}
which proves the desired assertion \eqref{p5-1-1}.
\end{proof}

\section{Application: Random Conductance Model}
We will apply results in the previous two sections to study heat kernel estimates and elliptic Harnack inequalities for
random conductance models on
$\mathbb{L}:=\mathbb{Z}^{d_1}_+\times\mathbb{Z}^{d_2}$ (with
$d_1, d_2\in \mathbb{Z}_+$ such that $d_1+d_2\ge1$)
with stable-like jumps.

Let $V=\mathbb{L}$
and
$\{w_{x,y}(\w): x,y\in \mathbb{L}\}$
be a sequence of independent
(but not necessarily identically distributed) random variables on some probability space
$(\Omega, \F_\Omega, \Pp)$ such that
$w_{x,y}=w_{y,x}\ge 0$ for any $x\neq y$, and $w_{x,x}=0$ for any $x\in \mathbb{L}$.
Let $\mu$ be a strictly positive (random) measure on $\mathbb{L}$.
For $\Pp$-a.s.\ fixed $\omega \in \Omega$, we consider the following regular Dirichlet form $(D^{\w},\F^{\w})$
on $L^2(\mathbb{L};\mu)$,
\begin{align*}
D^{\w}(f,f)&=\sum_{x,y\in \mathbb{L}}(f(x)-f(y))^2\frac{w_{x,y}(\w)}{|x-y|^{d+\alpha}},
\quad f\in \F^{\w},\\
\F^{\w}&=\{f\in L^2(\mathbb{L};\mu): D^{\w}(f,f)<\infty\}.
\end{align*}
Note that unlike \eqref{e1-1} we do not include the term $\mu_x\mu_y$ in the Dirichlet form above, but it is obviously reduced into \eqref{e1-1} by replacing $w_{x,y}$ with $\frac{w_{x,y}(\w)}{\mu_x\mu_y}$. We prefer to the expression above due to the consistency of notations as those in \cite[Section 5.2.1]{CKW}.
Let $(X_t^{\w})_{t\ge 0}$ be the symmetric Hunt process associated with $(D^{\w},\F^{\w})$, whose infinitesimal generator ${\mathcal L}^{\w}$ is given by
 \begin{equation}\label{e5-1}
 {\mathcal L}^{\w}f(x):=\frac{1}{\mu_x}\sum_{y\in \mathbb{L}}\big(f(y)-f(x)\big)\frac{w_{x,y}(\w)}{|x-y|^{d+\alpha}}.
 \end{equation}
In the literature, when
$\mu$ is the counting measure (resp. $\mu_x=\mu^{\w}_x:=\sum_{z\in \mathbb{L}}\frac{w_{x,z}(\w)}{|x-z|^{d+\alpha}}$ for all $x\in V$), $X$ is called
a variable speed random walk (resp. a constant speed random walk),
and denote by $p^\w(t,x,y)$ (resp. $q^{\w}(t,x,y)$) the heat kernel of the corresponding process $X$.
\subsection{Heat kernel estimates and local limit theorem}

\begin{theorem}\label{t3-1}{\bf (Heat kernel estimates for variable speed random walks)}
Suppose that $d>4-2\alpha$,
\begin{equation}\label{t3-1-0a}
\sup_{x,y\in \mathbb{L}: x\neq y}\Pp\big(w_{x,y}=0\big)<2^{-4}
\end{equation} and
\begin{equation}\label{t3-1-0}
\sup_{x,y\in \mathbb{L}: x\neq y}\Ee[w_{x,y}^{2p}]<\infty~~~\mbox{and}~~~\sup_{x,y\in
\mathbb{L}: x\neq y}\Ee[w_{x,y}^{-2q}\I_{\{w_{x,y}\neq 0\}}]<\infty
\end{equation}
for $p,q\in \Z_{+}$ with
$$p>\max\left\{\frac{d+1+\theta_0}{d\theta_0 },\frac{d+1}{2\theta_0(2-\alpha)}\right\} \,\, \text{ and    }\,\,
q>\frac{d+1+\theta_0}{d\theta_0},$$
where {$\theta_0:={\alpha}/({2d+\alpha})$}.
Then,
for {$\Pp$-a.s.\ $\w\in \Omega$ and every $x\in \mathbb{L}$},
there is a constant
$R_x(\w)\ge1$ such that for all $R>R_x(\w)$, $t>0$ and $y\in \mathbb{L}$ satisfying $t\ge (|x-y|\vee R_x(\w))^{\theta\alpha}$,
\begin{equation}\label{t3-1-1}
C_1\Big(t^{-d/\alpha}\wedge \frac{t}{|x-y|^{d+\alpha}}\Big)\le p^{\w}(t,x,y)
\le C_2\Big(t^{-d/\alpha}\wedge \frac{t}{|x-y|^{d+\alpha}}\Big),
\end{equation}
where $C_1,C_2>0$ and $\theta\in(0,1)$ are constants independent of $t$, $x$ and $y$.
\end{theorem}
\begin{proof}
Obviously, the counting measure $\mu$ satisfies
Assumptions {\bf($d$-Vol)}.
Under \eqref{t3-1-0a}, \eqref{t3-1-0} and the condition that
$\{w_{x,y}(\w): x,y\in \mathbb{L}\}$
is a sequence of independent random variables, we can follow the proof of
\cite[Proposition 5.6]{CKW}
(in particular, the Borel-Cantelli arguments), and prove that
there is a constant $\theta\in (0, {\alpha}/({2d+\alpha}))$ such that for all fixed
$x\in \mathbb{L}$ and $\Pp$-a.s. $\w\in \Omega$,
Assumptions {\bf{(HK1)}}, {\bf (HK2)}
and {\bf(HK3)} hold for random conductance $\{w_{x,y}(\w): x,y \in \mathbb{L}\}$
with the associated center $0$ and constant $R_0$ being replaced by $x$ and $\widetilde R_x(\w)\ge1$ respectively. Therefore, according to Theorem \ref{t2-2} and Theorem \ref{t2-3},
for every $x\in \mathbb{L}$ and $\Pp$-a.s. $\w\in \Omega$, there exists $R_x(\w)\ge1$ such that \eqref{t3-1-1} is fulfilled.
\end{proof}

\begin{theorem}\label{t3-3} {\bf (Heat kernel estimates for constant speed random walks)}
Suppose that $d>4-2\alpha$, $w_{x,y}>0$ for all $x\neq y \in \mathbb{L}$,
\begin{equation}\label{t3-3-0}
\inf_{x\in \mathbb{L}}{\mu_x^\w>0},
\quad\sup_{x\in \mathbb{L}}{\mu_x^\w<\infty}
\end{equation}
and \eqref{t3-1-0} holds with the same constants $p,q\in \Z_+$ as these in Theorem $\ref{t3-1}$.
Then
for {$\Pp$-a.s.\ $\w\in \Omega$ and every $x\in \mathbb{L}$},
there is a constant
$R_x(\w)\ge1$
such that for all $R>R_x(\w)$ and for all $t>0$ and $y\in \mathbb{L}$ satisfying $t\ge (|x-y|\vee R_x(\w))^{\theta_1\alpha}$,
\begin{equation}\label{t3-3-1}
C_3\Big(t^{-d/\alpha}\wedge \frac{t}{|x-y|^{d+\alpha}}\Big)\le q^{\w}(t,x,y)
\le C_4\Big(t^{-d/\alpha}\wedge \frac{t}{|x-y|^{d+\alpha}}\Big),
\end{equation}
where $C_3,C_4>0$ and $\theta\in(0,1)$ are constants independent of $t$, $x$ and $y$.
\end{theorem}
\begin{proof}
For the constant speed random walk, $\mu_x^{\w}=\sum_{z\in \mathbb{L}}\frac{w_{x,z}(\w)}{|x-z|^{d+\alpha}}$. Then, the associated Dirichlet form enjoys the expression \eqref{e1-1} with $\mu_x=\mu_x^\w$ and $w_{x,y}=\frac{w_{x,y}(\w)}{\mu_x^\w\mu_y^\w}$.
Under \eqref{t3-3-0},
there are constants $0<c_1\le c_2<\infty$ such that for all $x\in \mathbb{L}$, $c_1\le \mu^{\w}_x\le c_2$, which
implies that Assumption {\bf($d$-Vol)} holds.
Due to the fact that $\mu^{\w}_x$ is uniformly bounded from upper and below, we can follow in the proof of Theorem
\ref{t3-1} to verify that for all fixed
$x\in \mathbb{L}$ and $\Pp$-a.s. $\w\in \Omega$,
Assumptions {\bf{(HK1)}}, {\bf (HK2)}
and {\bf(HK3)} hold
with the associated center $x$ and  constant $\widetilde R_x(\w)\ge1$ respectively.
Therefore, the desired assertion follows from Theorem \ref{t2-2} and Theorem \ref{t2-3}.
\end{proof}

\begin{remark} (1)
Similar
to the case for variable speed random walks (see Theorem \ref{t3-1}), we can allow the conductance in Theorem \ref{t3-3} to be degenerate; that is,
Theorem \ref{t3-3} still holds,
even if the assumption that $w_{x,y}>0$ for all $x,y\in \mathbb L$ is replaced by
an upper bound of the probability
for degenerate conductances
like \eqref{t3-1-0a}.
As seen from the proof of \cite[Proposition 5.6]{CKW},  such an explicit
upper bound depends on the constants $c_1,c_2$ for uniform bounds of $\{\mu_x^\w\}_{x\in \mathbb{L}}$.
We omit the details here.

(2) For the nearest neighbor random conductance models (see e.g. \cite{ABDH,BD,MatP}),
percolation estimates are crucially used to remove or weaken the condition \eqref{t3-3-0}. However, such estimates are not available (and are not easy to gain at least) for our model. This explains the reason why we need assume \eqref{t3-3-0} in Theorem \ref{t3-3}.
\end{remark}

According to \cite[Theorem 1]{CH} (see \cite[Section 4]{BH} or \cite[Theorem 1.11]{ADS2} for related discussions), we have the following local limit theorem for $(X_t^\w)_{t\ge0}$.
For any $a>0$, let $k_{a,t}(x):=k_{a,t}(0,x)$ be the transition density function corresponding to symmetric $\alpha$-stable processes with L\'evy measure $a|z|^{-d-\alpha}\,dz$.
\begin{theorem}[{\bf Local limit theorem for variable speed random walks}]\label{p3-2}
 Under the setting of Theorem $\ref{t3-1}$, assume further that $\Ee {w_{x,y}}=a$ for all $x\neq y\in \mathbb{L}$. Then, for any $T_2>T_1>0$ and $k>1$,
$$\lim_{n\to\infty}\sup_{|x|\le k}\sup_{t\in [T_1,T_2]} |n^dp^{\w}(n^\alpha t, 0, [nx])-k_{a,t}(x)|=0,$$ where $[x]=([x_1], \cdots, [x_d])$ for any $x=(x_1,\cdots, x_d)\in \R_+^{d_1}\times \R^{d_2}$. \end{theorem}

\begin{proof} We consider the scaling limit procedure as used in \cite[Section 4]{CKW}, and adopt the notations in \cite{CH}.
Let $E=F=\R_+^{d_1}\times \R^{d_2}$, $G=\mathbb{L}$ with $d(x,y)=|x-y|$,  and $G^n=n^{-1}\mathbb{L}$ with $d_{G^n}(x,y)=nd(x,y)$.
Denote by $\nu$ the Lebesgue measure, and by $\nu^n$ the counting measure on $G^n$. It is clear that conditions (a)--(c) in \cite[Assumption 1]{CH} hold with $\alpha(n)=n$, $\beta(n)=n^d$ and $\gamma(n)=n^\alpha$.
Let $X^{n,\w}_t:=n^{-1}X^{\w}_{n^{\alpha}t}$ for any $t>0$, and denote by $p^{n,\w}(t,x,y)$
its heat kernel on $G^n=n^{-1}\mathbb{L}$ with respect to the counting measure $\nu^n$. Then, $p^{n,\w}(t,x,y) =p^{\w}(n^{\alpha}t,nx,ny)$ for all
$x,y \in n^{-1}\mathbb{L}$ and $t>0$, and so it suffices to prove that the conclusion of \cite[Theorem 1]{CH} holds for
 $p^{n,\w}(t,x,y)$. Note that, though
\cite[Theorem 1]{CH} is stated for the local limit theorem for a discrete time Markov process, by carefully tracking the argument,
we can verify that the proof also works for a continuous time Markov process, see the proof of \cite[Theorem 1.11]{ADS2} for more details. Now, we are going to check that (d) in \cite[Assumption 1]{CH} holds for
$p^{n,\w}(t,x,y)$.

According to \cite[Theorem 1.1]{CKW}, under the setting of Theorem $\ref{t3-1}$, the quenched invariance principle holds for
 $(X^{\w}_{t})_{t\ge0}$ with the limit process being a
 (reflected) symmetric $\alpha$-stable L\'evy process
 on $\R_+^{d_1}\times \R^{d_2}$
 with jumping measure $a|z|^{-d-\alpha}\,dz$.
 In particular, this implies that for every $f \in C_b(\R_+^{d_1}\times \R^{d_2})$,
 $t>0$ and $x\in \R_+^{d_1}\times \R^{d_2}$,
 \begin{equation*}
 \lim_{n\rightarrow \infty}
 \left|\mathbf{E}^{n,\w}_{[x]_n}\big[f(X^{n,\w}_t)\big]-\int_{\R^d} f(z)k_{a,t}(x-z)\,dz\right|=0,
 \end{equation*}
where $[x]_n:=n^{-1}[x]$ and $\mathbf{E}^{n,\w}_x$ denotes the expectation with respect to the law of process $(X^{n,\w}_t)_{t\ge 0}$
 with the initial point $x \in n^{-1}\mathbb{L}$.
Combining this with Theorem \ref{p2-2} (which was used to estimate the term $\big|\mathbf{E}^{n,\w}_{[x]_n}[f(X^{n,\w}_t)]
-\mathbf{E}^{n,\w}_{[x]_n}[f(X^{n,\w}_s)]\big|$ for $0<s<t$), we can further verify that
  for every $f \in C_b(\R_+^{d_1}\times \R^{d_2})$, $0<T_1<T_2$ and $x \in \R_+^{d_1}\times \R^{d_2}$,
 \begin{equation}\label{p4-2-1}
 \lim_{n\rightarrow \infty} \sup_{t\in [T_1,T_2]}
 \left|\mathbf{E}^{n,\w}_{[x]_n}\big[f(X^{n,\w}_t)\big]-\int_{\R^d} f(z)k_{a,t}(x-z)\,dz\right|=0.
 \end{equation}
 Since for every $x_0\in \R_+^{d_1}\times \R^{d_2}$ and $r>0$,
 $\displaystyle\int_{\partial B(x_0,r)}k_{a,t}(x-z)\,dz=0$, \eqref{p4-2-1} in turn yields that
 $$ \lim_{n\rightarrow \infty} \sup_{t\in [T_1,T_2]}
 \left|\mathbf{P}^{n,\w}_{[x]_n}\big(X^{n,\w}_t\in B(x_0,r)\big)-\int_{B(x_0,r)} f(z)k_{a,t}(x-z)\,dz\right|=0, $$ see e.g. \cite[Theorem 2.1]{Bill}.
Therefore, (d) in \cite[Assumption 1]{CH} is satisfied.

On the other hand, it follows from Theorem \ref{p2-2} again
and the on-diagonal upper bound
(see \eqref{t3-1-1}) of $p^{\w}(t,x,y)$, there  are constants $\delta\in (0,1)$  and $\beta \in (0,1)$ such that for any $0<T_1<T_2$, $r>0$,
$\gamma\in (0,r]$ and $n\ge1$ large enough,
\begin{equation}\label{p4-2-2}
\begin{split}\sup_{x,y\in B_{G^n}(0,\alpha(n)r),\atop (n\gamma)^{\delta}\le d_{G_n}(x,y)\le n \gamma}\!\!
\sup_{t\in [T_1,T_2]}|p^{\w}(n^\alpha t,0,[nx])\!-\! p^{\w}(n^\alpha t,0,[ny])|\!\le\! c_1(T_1,T_2,r)n^{-d}\gamma^{\beta},\end{split}
\end{equation}
where $c_1>0$ is independent of $n$ and $\gamma$,
and $B_{G^n}(0,r)$ is denoted by the ball on $G^n$ with center $0$ and radius $r$. Note that, \eqref{p4-2-2} is slightly weaker than
\cite[Assumption 2]{CH}, since we require
to take the supremum of $(n\gamma)^{\delta}\le d_{G_n}(x,y)$.
However, according to \eqref{p4-2-2} and
 the on-diagonal upper bound of
 heat kernel \eqref{t3-1-1}, we can get that for every $x \in \R_+^{d_1}\times \R^{d_2}$, $t\in [T_1,T_2]$,
  each fixed $\delta>0$
   and $r\in (0,\delta]$ small enough,
\begin{align*}
|J_1(t,n,x,r)|:&=n^{-d}\left|\sum_{y \in G^n: y \in B(x,r)}\left(p^{n,\w}(t,0,y)-p^{n,\w}(t,0,x)\right)\right|\\
&\le n^{-d}\left|\sum_{y \in G^n: (nr)^{\delta}\le d_{G_n}(y,x)\le nr}\left(p^{n,\w}(t,0,y)-p^{n,\w}(t,0,x)\right)\right|\\
&\quad +n^{-d}\left|\sum_{y \in G^n:  d_{G_n}(y,x)\le (nr)^{\delta}}\left(p^{n,\w}(t,0,y)-p^{n,\w}(t,0,x)\right)\right|\\
&\le n^{-d}\sup_{(nr)^{\delta}\le d_{G_n}(y,x)\le nr}\left|p^{n,\w}(t,0,y)-p^{n,\w}(t,0,x)\right|\cdot
\nu^n\big(B(x,r)\big)\\
&\quad +2n^{-d}\sup_{d_{G_n}(y,x)\le (nr)^{\delta}}|p^{n,\w}(t,0,y)|\cdot \nu^n\big(B(x, n^{-1}(nr)^{\delta})\big)\\
&\le c_2(T_1,T_2,\delta)n^{-d}\Big[\nu^n\big(B(x,r)\big)n^{-d}
r^\beta +n^{-d} \nu^n\big(B(x, n^{-1+\delta}r^{\delta})\big) \Big]\\
&\le c_3(T_1,T_2,\delta)n^{-d}\Big[ r^{d+\beta}+ n^{-(1-\delta)d}r^{\delta d}\Big].
\end{align*}
Hence
$$ \lim_{r\to 0}\limsup_{n\rightarrow \infty}\sup_{t \in [T_1,T_2]}n^d|J_1(t,n,x,r)|=0.$$ With this estimate replacing that of $J_1(t,n)$ in the proof of
\cite[Theorem 4.2]{BH}, the proof of \cite[Theorem 1]{CH} is valid. See also the proof of \cite[Theorem 1.11]{ADS2}.

Therefore, the desired assertion follows from \cite[Theorem 1]{CH}. \end{proof}

\begin{remark}
Here
we take $V=\mathbb{L}$ in Theorem \ref{t3-1} and Theorem
\ref{p3-2} for simplicity. As in
\cite[Section 5]{CKW}, these results also hold for
more general
$d$-sets, including the Sierpinski gasket.
\end{remark}

\subsection{Two counterexamples}
\subsubsection{{\bf Heat kernel estimates}}
In this subsection, we give an example that  shows the heat kernel may behave anomalously
without the moment condition of $w_{xy}^{-1}$.
This example is heavily motivated by \cite[Theorem 2.1\,(1)]{BBHK}.

In the following, denote by
$P^{\w}(n,x,y)$ the $n$-step transition probability for discrete time Markov chain associated
with the conductance $\{C_{x,y}(\w):=\frac{w_{x,y}(\w)}{|x-y|^{d+\alpha}}: x,y\in \Z^d\}$, and by
$q^{\w}(t,x,y)$ the heat kernel for the constant speed random walk
associated with $\{C_{x,y}(\w):x,y\in \Z^d\}$.

\begin{proposition}\label{prop5112}
Let $d\ge 5$ and $\kappa> 1/d$. Then, there exist a sequence of independent random variables
$\{w_{x,y}: x,y\in \mathbb Z^d\}$ such that
$w_{x,y}\stackrel
{d}{=}w_{x',y'}$ for $|x-y|=|x'-y'|$,
and
\begin{equation}\label{eq:bbhk1}
\Pp(w_{x,y}\le 1)=1, ~~~~
c_1^{-1}\le \Ee[w_{x,y}]\le c_1,\quad
\forall x\ne y \in \Z^d,
\end{equation}
for some constant $c_1\ge 1$ independent of $x$ and $y$; while
\begin{equation}\label{eq:bbhk2}
P^\omega (2n, 0,0)\ge C(\omega)
\frac{e^{-(\log n)^\kappa}}{n^2},\quad
\forall n\ge 1,\mbox{a.s. } \Pp
\end{equation} for some random variable $C(\omega)>0$.
In particular, for any $\delta>0$ small enough, there exists a constant $C_1(\w)>0$ such that
$$
q^{\w}(t,0,0)\ge C_1(\w)t^{-2-\delta},\quad t\ge1.
$$
\end{proposition}
\begin{proof}
Since $$
q^{\w}(t,x,y)=\sum_{n\ge 0}\frac{t^n}{n!}e^{-t}P^{\w}(n,x,y),\quad t>0,\ x,y\in \Z^d,
$$
it suffices to prove \eqref{eq:bbhk2}.
The proof is similar to that of \cite[Theorem 2.1\,(1)]{BBHK} except that we should control long range bonds suitably.
For $\kappa>1/d$, choose $\eps>0$ small enough so that
$(1+(4d+2)\eps)/d<\kappa$.
Take a sequence of independent random variables $\{w_{x,y}: x,y\in \mathbb Z^d\}$
such that
\begin{itemize}
\item[(i)] for every
$|x-y|=1$, we set
\begin{equation}\label{eq:bbhkasqq}
\Pp(w_{x,y}=1)>p_c(d), ~~
\Pp(w_{x,y}=2^{-N})=c_1N^{-(1+\eps)},~~~~\forall N\ge 1,
\end{equation}
for some $c_1>0$, where $p_c(d)$ is the critical probability of the Bernoulli percolation on
$\Z^d$;
\item[(ii)] for $|x-y|>1$, choose $\{w_{x,y}: x,y\in \Z^d\}$ to satisfy \eqref{eq:bbhk1} and
\begin{equation}\label{eq:bbhkasm}
\Pp\left(\sum_{z\in \mathbb{L}:|y-z|>1}C_{y,z}\le M\right)=1,\quad~~
\Pp\left(\sum_{z\in \mathbb{L}: |y-z|>1}C_{y,z}\le 2^{-N}\right)\ge c_2N^{-\eps},~~~~\forall N\ge 1,
\end{equation}
for some $M\ge 1$ and $c_2>0$, where $C_{x,y}=w_{x,y}/|x-y|^{d+\alpha}$. \end{itemize}

Define
\[
\ell_N=N^{(1+(4d+2)\eps)/d}
\]
and for each $x\in \mathbb Z^d$, let
$A_N(x)$ be the event that the configuration
near $y=x+{\bf e}_1$ and $ z=x+2{\bf e}_1$ (here  ${\bf e}_1=(1,0,\cdots, 0)$) is as
follows:
\begin{itemize}
\item[(iii)] $C_{y,z}=1$, $C_{x,y}=2^{-N}$ and other nearest neighbor bonds (that is, bonds with length $1$;
n.n. bonds in short)
connected to $y$ and $z$
have conductance $\le 2^{-N}$; moreover,
$\sum_{u\in \mathbb{L}:|y-u|>1}C_{y,u}\le 2^{-N}$ and
$\sum_{u\in \mathbb{L}: |z-u|>1}C_{z,u}\le 2^{-N}$.
\item[(iv)] $x$ is connected to the boundary of the box of side length $(\log \ell_N)^2$
centered at $x$ by n.n. bonds and conductance $1$.\end{itemize}
Since bonds with conductance $1$ percolate,
using \eqref{eq:bbhkasqq} and \eqref{eq:bbhkasm},
we have
\[
{\Pp}(A_N(x))\ge
c_3 N^{-1-\eps}\cdot N^{-(4d-3)\eps}\cdot N^{-2\varepsilon}
=c_3N^{-(1+4d\eps)},\] where we used the fact that $\Pp(w_{x,y}\le 2^{-N})\le c_* N^{-\varepsilon}$ for all $x,y\in \Z^d$ with
$|x-y|=1$.
Now, let $\mathbb G_n$ be a grid of vertices in
$[-\ell_N,\ell_N]^d\cap {\Z^d}$ that are located by distance $2(\log \ell_N)^2$.
Then, the events $\{A_N(x): x\in \Z^d\}$ are independent, and
\[
{\Pp}\left(\bigcap_{x\in \mathbb G_n} A_N(x)^c\right)\le\exp\left(-c_4\left(\frac{\ell_N}{(\log \ell_N)^2}\right)^dN^{-(1+4d\eps)}\right)
\le e^{-c_5N^{\eps}},
\]
hence the intersection occurs only for finitely many $N$.

Given this stretched-exponential decay, we know that
every connected component of length $(\log \ell_N)^2$
in $[-\ell_N,\ell_N]^d\cap {\mathbb Z^d}$ is connected to the largest connected component in $[-2\ell_N,2\ell_N]^d\cap {\mathbb Z^d}$ by using only n.n. bonds for large enough $N$
(see \cite[Theorem 8.65]{Grim}). Hence, there exists
$N_0=N_0(\omega)<\infty$ such that for any $N\ge N_0$,
$A_N(x)$ occurs for some even-parity vertex
$x=x_N(\omega)\in [-\ell_N,\ell_N]^d\cap {\mathbb Z^d}$
which is connected to $0$ by a path (say Path$_N$) in
$[-2\ell_N,2\ell_N]^d\cap {\mathbb Z^d}$, on which
only the $N_0$ n.n. bonds close to the origin may have
conductance smaller than $1$.

Now suppose $N\ge N_0$ and let $n$ be such that
$2^N\le 2n< 2^{N+1}$. Let $x_N=x_N(\omega)$ be as above and $r_N$ be the length of Path$_N$. Let $\alpha=\alpha(\omega)$ be the minimum of the conductance in the n.n. bonds within $N_0$ steps from
the origin. Then the passage from the origin to $x_N$ in time $r_N$ has probability $\ge \alpha^{N_0} C_*^{-r_N}$
where $C_*=\sum_{|y|\ge 1} |y|^{-d-\alpha}>2d$  (due to the fact that $w_{x,y}\le 1$ for all $x,y\in \Z^d$), while the probability of staying on the bond $(y,z)$ for time $2n-2r_N-2$ is bounded from below by a constant which is independent of $\omega$. The transition probability across $(x,y)$ is of order
$2^{-N}$. Hence, we have
\[
P^\omega (2n,0,0)\ge c_6\alpha^{2N_0}C_*^{-2r_N}2^{-2N}.
\]
Using the comparison of the graph distance and the Euclidean distance (see \cite{AP}), we have
$r_N\le c\ell_N$ for $N$ large enough. Since
$n\asymp 2^N$, we obtain the desired estimate.

Finally, let us give a concrete example
that satisfies \eqref{eq:bbhkasm}. The first  equality in \eqref{eq:bbhkasm} is an easy consequence of  \eqref{eq:bbhk1}. Now, for any $l\ge 1$,
define $f_l(s)= (\log_2 (c_0/s))^{-\eps'l^{-1-d}}$ for
$s\le 1/2$ small enough ($\eps'$ is a small positive value chosen later).
For any $y,z\in \mathbb Z^d$ with $|y-z|=l$, let $\Pp(\w_{y,z}=1)\ge 1/2$ and $\Pp(\w_{y,z}\le s)=f_l(s)$ (by possibly choosing the constant $c_0$ and the range of $s$ in the definition of $f_l(s)$).
Then, we have
\begin{align*}
\Pp\left(\sum_{|y-z|>1}C_{y,z}\le 2^{-N}\right)&\ge \Pp\left(\bigcap_{l=2}^\infty
\bigcap_{|y-z|=l}\{w_{y,z}\le c'2^{-N}\}\right)\ge \prod_{l=2}^\infty f_l(c'2^{-N})^{cl^{d-1}}\\
&\ge c''\prod_{l=2}^\infty
N^{-\eps'l^{-1-d}\cdot cl^{d-1}}
=c''N^{-c\eps'\sum_{l=2}^\infty l^{-2}}
=c''N^{-c'''\eps'}.
\end{align*}
Choosing $\eps'=\eps/c'''$, we obtain
the inequality in \eqref{eq:bbhkasm}.
\end{proof}

\subsubsection{{\bf Elliptic Harnack inequalities}}
\label{EHI-cou}
Before presenting a counterexample such that\ \EHI\ does not hold for random conductance models with non-uniformly elliptic stable-like jumps, we give the following statement, which is a consequence of Proposition \ref{p5-1}.
\begin{proposition}\label{t5-1}
 Under the setting of Theorem $\ref{t3-1}$
 with $d_1=0$ $($i.e. $\mathbb{L}=\mathbb{Z}^{d}$$)$ and $d> \alpha$,
assume further that
 $\{w_{x,y}: x,y\in\Z^d\}$ is a sequence of independent random variables so that
 \begin{equation}\label{t5-1-0a}
 \sup_{x,y\in \Z^d:x\neq y} w_{x,y}^{-1}<\infty,\quad a.s. \,\,\Pp,
 \end{equation}
\begin{equation}\label{t5-1-0}
\inf_{x,y\in \Z^d: x \neq y}\Pp\big(w_{x,y}\le c_0\big)>0\quad\textrm{ for some }c_0>0,
\end{equation} and
\begin{equation}\label{t5-1-1}
\inf_{x,y\in \Z^d: x \neq y}\Pp\big(w_{x,y}>k\big)>0\quad\textrm{ for any }k>0.
\end{equation}
Then,  for $\Pp$-a.s. $\w\in \Omega$,  \ \EHI\
does not hold at any $x_0\in \Z^d$.
\end{proposition}
\begin{proof}
Without loss of generality, we only verify the case that $x_0=0$.
As explained in the proof of Theorem \ref{t3-1}, for $\Pp$-a.s. $\w\in \Omega$,
Assumptions {\bf (HK1)}, {\bf (HK2)} and {\bf (HK3)} hold with some constant $\theta'\in (0,{{\alpha}/({2d+\alpha})})$.  According to Proposition \ref{p5-1}, it suffices to disprove the inequality \eqref{p5-1-1}.

For every
fixed $R\ge1$, $N>0$ and $z \in B(0,4R)^c$, we set
$$
p(R,N,z):=\Pp\Big(w_{0,z}\le NR^{\alpha+\theta'(d-\alpha)}\sup_{v \in B(0,2R),v\neq 0}w_{v,z}\Big).
$$
Note that here $\theta'\in (0,{{\alpha}/({2d+\alpha})})$ is independent of $R$ and $N$.

Since $\{w_{x,y}: x,y\in\Z^d\}$ is a sequence of independent random variables, for the constant $c_0$ given in \eqref{t5-1-0} it holds that
\begin{align*}
p(R,N,z)&=1-\prod_{v\in B(0,2R),v\neq 0}\Pp\Big(w_{0,z} > NR^{\alpha+\theta'(d-\alpha)} w_{v,z}\Big)\\
&\le 1-\prod_{v\in B(0,2R),v\neq 0}\Pp\Big(w_{0,z} > NR^{\alpha+\theta'(d-\alpha)}c_0, w_{v,z}\le c_0\Big)\\
&=1-\prod_{v\in B(0,2R),v\neq 0}\Pp\Big(w_{0,z} > NR^{\alpha+\theta'(d-\alpha)}c_0\Big)\Pp\Big(w_{v,z}\le c_0\Big)\\
&\le 1-c_1(R,N)^{R^d}:=c_2(R,N).
\end{align*}
Here, by \eqref{t5-1-0} and \eqref{t5-1-1}, $c_1(R,N)$ and $c_2(R,N)$ are positive constants which can be chosen to be independent of $z$. Therefore, for every fixed $N$ and $R$, and any $k_0\ge1$,
\begin{align*}
&\sum_{k=k_0}^{\infty}\Pp\Big(\bigcap_{z\in \Z^d: 2^{k}< |z|\le 2^{k+1}}
\Big\{w_{0,z}\le NR^{\alpha+\theta'(d-\alpha)}\sup_{v \in B(0,2R),v\neq 0}w_{v,z}\Big\}\Big)\\
&\le \sum_{k=k_0}^{\infty}\prod_{z\in \Z^d: 2^{k}< |z|\le 2^{k+1}}p(R,N,z)\le \sum_{k=k_0}^{\infty}
c_2(R,N)^{2^{k}}<\infty.
\end{align*}

This along with the Borel-Cantelli lemma yields that for every $R\ge1$, $N>0$ and $\Pp$-a.s. $\w\in \Omega$,
there exists a constant $k_1:=k_1(\w,R,N)>0$ such that for all
$k\ge k_1$, we can find $z \in \Z^d$ with $2^k<|z|\le 2^{k+1}$ satisfying that
$$
w_{0,z}>NR^{{\alpha+\theta'(d-\alpha)}}\sup_{v \in B(0,2R),v\neq 0}w_{v,z}.
$$
In particular, \eqref{p5-1-1} does not hold true. The proof is finished.
\end{proof}

\begin{example}[{\bf Counterexample for $\EHI$}]\label{exmephi} Let $\{w_{x,y}:x,y\in \Z^d\}$ be a sequence of i.i.d. random variables on some probability space
$(\Omega, \F, \Pp)$ such that
for $w_{x,y}=w_{y,x}>0$ for any $x\neq y$ and $w_{x,x}=0$ for any $x\in \Z^d$.
Suppose that $d>\alpha$. Let $p>0$ be that constant given in Theorem \ref{t3-1}, and let $\varepsilon>0$.
 We assume that for any $x\neq y\in \Z^d$
$$\Pp(w_{x,y}=k)=k^{-(2p+1+\varepsilon)},\quad k\in \Z_+\,\, {\rm with }\,\,k\ge 2,$$ and that $$\Pp(w_{x,y}=1)=1-\sum_{k=2}^\infty \Pp(w_{x,y}=k).$$  Then, assumptions in Theorem \ref{t3-1} and Proposition \ref{t5-1} are fulfilled. Therefore, according to Proposition \ref{t5-1}, with this choice of $\{w_{x,y}\}$ the $\EHI$ does not hold for any $x_0\in \Z^d$.  \end{example}

\begin{remark}
Concerning nearest neighbor models, the equivalence between heat kernel estimates and parabolic Harnack inequalities, both of which only hold for sufficiently large balls,  was established in \cite[Theorem 1.10]{BChen}. However, the proof (see that of \cite[Theorem 7.2]{BChen}) crucially uses some priori estimates of heat kernel for small time, see \cite[(3.9)]{BChen}. As we mentioned above, such estimates are not available in the present setting.
\end{remark}

We close this section with one remark on weak elliptic Harnack inequalities $(\WEHI)$ for random conductance models with stable-like jumps. The reader can refer to \cite{CKW3} and references therein for more details on $\WEHI$ for non-local Dirichlet forms.
\begin{remark} Under the setting of Theorem $\ref{t3-1}$
 with $d_1=0$ (i.e. $\mathbb{L}=\mathbb{Z}^{d}$) and $d>\alpha$, if \eqref{t5-1-0a} holds, then  for $\Pp$-a.s. $\w\in \Omega$ and  every $x\in \Z^d$, there exists a constant $R_x(\w)\ge1$ such that
for any $R>R_x(\w)$ and positive harmonic function $h$ on $B(x,2R)$,
it holds that
\begin{equation}\label{e:ehi}
\frac{1}{R^d}\sum_{z\in B(x,R)}h(z)\le c_0\inf_{z\in B(x,R)}h(z),
\end{equation}
where $c_0>0$ is a non-random constant independent of $x, R_x(\w), R$ and $h$.
Note that, by Proposition \ref{t5-1},
\EHI\ does not hold under \eqref{t5-1-0a}--\eqref{t5-1-1}, and so Example \ref{exmephi} provides us a concrete example that
\WEHI\ holds but \EHI\ fails for random conductance models with stable-like jumps.

The conclusion \eqref{e:ehi} mainly follows from  Moser's iteration procedures as in the proof of
\cite[Theorem 1.1]{FK} (see also the proof of \cite[Theorem 4.1]{DK1}), which are based on the following three ingredients.

First, under \eqref{t5-1-0a}, it is a direct consequence of the discrete fractional Sobolev
inequality on $\Z^d$ (with $d> \alpha$) that for all $R\ge 1$ and $x\in \Z^d$,
\begin{align*}
\Big(\sum_{z\in B(x,R)}f^{{2d}/({d-\alpha})}(z)\Big)^{({d-\alpha})/{d}}\le &c_1
\sum_{y,z\in B(x,R)}\big(f(x)-f(y)\big)^2\frac{w_{y,z}(\w)}{|y-z|^{d+\alpha}}\\
&
+c_1R^{-\alpha}\sum_{z\in B(x,R)}f^2(z),
\end{align*} where $c_1>0$ is independent of $x, R$ and $f$.
Second, also due to \eqref{t5-1-0a}, we can follow the arguments in \cite[Corollary 5]{DK} and use the scaling
property to establish the following weighted Poincar\'e inequality
\begin{align*}
&\sum_{z\in B(x,3R/2)}\big(
f(z)-f_{\Psi_R}\big)^2\Psi_R(z)\\
&\le  c_2R^{\alpha}\sum_{y,z\in B(x,3R/2)}
\big(f(x)-f(y)\big)^2\frac{w_{y,z}(\w)}{|y-z|^{d+\alpha}}\Psi_R(y)\wedge \Psi_R(z),
\end{align*}
where $\Psi_R(x):=\big((3R/2)-|x|\big)\wedge R$, $$f_{\Psi_R}:=
\Big(\sum_{z\in B(x,3R/2)}\Psi_R(z)\Big)^{-1}\Big(\sum_{z\in B(x,3R/2)}f(z)\Psi_R(z)\Big)$$ and $c_2$ is independent of $x, R$ and $f$.
Third, as explained in the proof of Theorem \ref{t3-1},
under \eqref{t3-1-0}, we know that for
$\Pp$-a.s. $\w\in \Omega$ and every $x\in \Z^d$, there exists a constant
$R_x(\w)\ge1$ such that for all $R>R_x(\w)$ and $R^{\theta}\le r \le 2R$,
$$\sup_{y\in (0,2R)}\sum_{z\in \Z^d: |y-z|\le r}\frac{w_{x,y}(\w)}{|z-y|^{d+\alpha-2}}
\le c_3r^{-\alpha}$$ and
$$\sup_{y\in (0,2R)}\sum_{z\in \Z^d: |y-z|> r}\frac{w_{x,y}(\w)}{|z-y|^{d+\alpha}}
\le c_3r^{-\alpha},$$
where $\theta\in (0,1)$ are $c_3>0$ are independent of $x\in \Z^d$, $R_x(\w)$, $r$ and $R$.

We finally emphasize that the argument above is based on
the condition \eqref{t5-1-0a}. It is an
interesting question to prove fractional Sobolev inequalities and weighted Poincar\'e
inequalities as above under weaker condition than \eqref{t5-1-0a}.
\end{remark}

\noindent {\bf Acknowledgements.}
The research of Xin Chen is supported by the National Natural Science Foundation of China (No.\ 11501361).\ The research of Takashi Kumagai is supported
by JSPS KAKENHI Grant Number JP17H01093 and by the Alexander von Humboldt Foundation.\
The research of Jian Wang is supported by the National
Natural Science Foundation of China (No.\ 11522106), the Fok Ying Tung
Education Foundation (No.\ 151002), and the Program for Probability and Statistics: Theory and Application (No.\ IRTL1704).

\end{document}